\documentclass[british]{article}
\usepackage{ae,aecompl}
\usepackage[T1]{fontenc}
\usepackage[latin9]{inputenc}
\usepackage{parskip}
\usepackage{mathtools}
\usepackage{amsmath}
\usepackage{amsthm}
\usepackage{amssymb}
\usepackage{geometry}
\makeatletter
\theoremstyle{plain}
\newtheorem{thm}{\protect\theoremname}
\theoremstyle{definition}
\newtheorem{defn}[thm]{\protect\definitionname}
\theoremstyle{remark}
\newtheorem{rem}[thm]{\protect\remarkname}
\theoremstyle{plain}
\newtheorem{prop}[thm]{\protect\propositionname}
\theoremstyle{plain}
\newtheorem{lem}[thm]{\protect\lemmaname}
\theoremstyle{plain}
\newtheorem{cor}[thm]{\protect\corollaryname}

\usepackage{capt-of}
\usepackage{bbm}
\date{}

\theoremstyle{definition}
\newtheorem{notation}[thm]{Notation}
\makeatother

\usepackage{babel}
\providecommand{\corollaryname}{Corollary}
\providecommand{\definitionname}{Definition}
\providecommand{\lemmaname}{Lemma}
\providecommand{\propositionname}{Proposition}
\providecommand{\remarkname}{Remark}
\providecommand{\theoremname}{Theorem}

\begin{document}
\global\long\def\IN{\mathbb{N}}%
\global\long\def\II{\mathbbm{1}}%
\global\long\def\IZ{\mathbb{Z}}%
\global\long\def\IQ{\mathbb{Q}}%
\global\long\def\IR{\mathbb{R}}%
\global\long\def\IC{\mathbb{C}}%
\global\long\def\IP{\mathbb{P}}%
\global\long\def\IE{\mathbb{E}}%
\global\long\def\IV{\mathbb{V}}%

\title{Reconstruction of the Probability Measure and the Coupling Parameters
in a Curie-Weiss Model}
\author{Miguel Ballesteros\thanks{IIMAS-UNAM, Mexico City, Mexico}, Ramsés
H. Mena\footnotemark[1], Arno Siri-Jégousse\footnotemark[1], and
Gabor Toth\footnotemark[1]}
\maketitle
\begin{abstract}
\noindent The Curie-Weiss model is used to study phase transitions
in statistical mechanics and has been the object of rigorous analysis
in mathematical physics. We analyse the problem of reconstructing
the probability measure of a multi-group Curie-Weiss model from a
sample of data by employing the maximum likelihood estimator for the
coupling parameters of the model, under the assumption that there
is interaction within each group but not across group boundaries.
The estimator has a number of positive properties, such as consistency,
asymptotic normality, and exponentially decaying probabilities of
large deviations of the estimator with respect to the true parameter
value. A shortcoming in practice is the necessity to calculate the
partition function of the Curie-Weiss model, which scales exponentially
with respect to the population size. There are a number of applications
of the estimator in political science, sociology, and automated voting,
centred on the idea of identifying the degree of social cohesion in
a population. In these applications, the coupling parameter is a natural
way to quantify social cohesion. We treat the estimation of the optimal
weights in a two-tier voting system, which requires the estimation
of the coupling parameter.
\end{abstract}


\section{Introduction}

Models of ferromagnetism have long served as foundational tools in
statistical mechanics, enabling the exploration of collective behaviour
in systems with many interacting components. As with many other physics
models, they have also attracted the attention of mathematical physicists
who have contributed rigorous results which frequently confirm the
physicists' intuition. Among the most celebrated is the Ising model,
introduced in the early 20th century by Wilhelm Lenz and his student
Ernst Ising \cite{Ising1925}. In its classical form, the Ising model
consists of a lattice of binary spins or magnets, where each spin
interacts with its nearest neighbours and aligns in response to an
external field and thermal fluctuations. Although the one-dimensional
version exhibits no phase transition at finite temperature, higher-dimensional
cases, such as the two-dimensional model solved by Onsager \cite{Onsager1944},
exhibit rich phase behaviour and critical phenomena, making the Ising
model a central object of study in both physics and applied mathematics.

The Curie-Weiss model, introduced as a mean-field approximation of
the behaviour of ferromagnets \cite{Weiss1907}, simplifies the spatial
complexity by assuming that each spin interacts equally with all others.
This global interaction structure permits a rigorous mathematical
analysis of phase transitions and critical behaviour (see \cite{Ell1985}
for an in-depth treatment of the model). Beyond its physical origins,
the Curie-Weiss model has been widely used in other domains, including
economics \cite{BD2001}, political science \cite{Ki2007}, and sociology
\cite{CGh2007}. An extension of the Curie-Weiss model is the specification
that instead of having a homogeneous population of interacting agents,
there are several identifiable groups which have distinct cultures
or attitudes, manifesting in different voting decisions. The first
version of a multi-group Curie-Weiss model was introduced in \cite{CGh2007};
subsequently, similar models were studied in \cite{BRS2019,LoweSchu2020,LSV2020,KirsToth2020,KirsToth2020b,KirsToth2022b},
including the statistical problem of community detection \cite{BRS2019,LoweSchu2020,BaMePeTo2023}.

One particularly compelling application of the Curie-Weiss model lies
in the domain of collective decision-making, such as voting. In this
context, each spin is interpreted as an individual voter casting a
binary `yes' or `no' vote. The mutual influence between voters,
akin to the ferromagnetic alignment in physical systems, can model
peer pressure, shared information, or ideological affinity. This analogy
becomes especially fruitful in the analysis of two-tier voting systems,
where decisions are made in two stages: individuals vote to determine
the outcome of a local group (e.g., a state or district), and these
local outcomes are then aggregated at a higher level (e.g., in a council
or federal assembly). The central question in this setting concerns
the \emph{assignment of optimal voting weights} to the representatives
of each group, ensuring fair representation in the face of population
imbalances or correlated voting behaviour (see \cite{FelsMach1998}
for a treatment of voting power).

This challenge is classically illustrated in institutions like the
Council of the European Union or the United Nations Security Council.
However, its significance extends far beyond geopolitics. In modern
automated systems, such as recommendation engines, online platforms
aggregating user preferences, and AI-based decision frameworks, users'
preferences often exhibit correlated structures. Here too, the need
arises to determine how individual or subgroup preferences should
be weighted when computing global outcomes. As such, tools from statistical
mechanics, and in particular \emph{statistical estimation methods}
applied to the Curie-Weiss model, provide both a conceptual and practical
bridge between physical models and real-world decision-making systems.

The present article contributes to this line of inquiry by addressing
the statistical estimation of the coupling parameters $\beta_{\lambda}$
in a multi-group Curie-Weiss model, a quantity which modulates the
strength of interaction between agents or voters within each of the
$M\in\IN$ groups indexed by $\lambda\in\IN_{M}$\footnote{We will write $\IN_{m}\coloneq\left\{ 1,\ldots,m\right\} $ for any
$m\in\IN$}. Estimating $\beta_{\lambda}$ from observed voting data yields critical
insights into the degree of collective behaviour or ideological alignment,
which can, in turn, inform the derivation of optimal weights in two-tier
voting systems. Section \ref{sec:Optimal-Weights} of this article
explores this application in detail under the assumption that voters
interact within each group but not across group boundaries, connecting
theoretical results to policy-relevant and business-related settings.
Imagine a population which is divided into $M$ groups along cultural
or national lines. An example is the European Union with its 27 member
states. Each member state sends a representative to the Council of
the European Union, where votes take place according to a set of fairly
complicated rules. We consider the simpler case of a weighted voting
system: each member state has a certain voting weight, and its representative's
vote in the council is multiplied by the voting weight. Examples of
weighted voting systems are `one person, one vote' in popular votes
or voting weights proportional to the population of each country in
a council. By imposing a `fairness criterion,' we can determine
what the weights in the council ought to be. The theoretical optimal
weights may depend on the underlying voting model (such as the Curie-Weiss
model treated in this article) which describes the population's voting
behaviour in probabilistic terms. If we reconstruct the underlying
voting model, we also obtain estimators for the optimal voting weights,
which can then be calculated from a sample of observations.

Prior work has been done on the estimation problem of parameters of
spin models such as those we mentioned. A classical reference is the
book \cite{Kastelei1956}. The estimation of the interaction parameter
from data belonging the realm of the social sciences was studied in
\cite{GalBarCo2009}. Work has also been done on the estimation of
parameters in more complicated models referred to as spin glass models
with random interaction structure (see \cite{Chatterj2007,CheSenWu2024}).

Our methodology is grounded in \emph{maximum likelihood estimation},
which is a natural and widely used approach in parametric inference,
provided we have a good reason to assume we know the underlying statistical
model which generated the data. The maximum likelihood estimator for
the coupling parameter  is particularly attractive because it is \emph{well-defined
for all values of $\beta_{\lambda}$}, including for parameter values
close to the critical regime (and typical samples from this distribution).
Furthermore, the maximum likelihood estimator enjoys desirable statistical
properties: it is \emph{consistent}, meaning it converges in probability
to the true parameter value as the number of observations increases.
It is \emph{asymptotically normal}, allowing for the construction
of confidence intervals and hypothesis tests. The estimator satisfies
\emph{large deviation principles}, providing robust upper bounds on
the probability of estimation errors in finite samples.

The main computational drawback of this approach is the necessity
to compute the normalisation constant (called partition function in
statistical physics) of the Curie-Weiss model, which scales exponentially
with the number of voters. This challenge is well known in the literature
and has inspired a range of approximation techniques. The path of
further research will lead us to consider maximum likelihood estimators
based on an approximation to the maximum likelihood optimality condition
(\ref{eq:opt}), which sidesteps the costly calculation of the partition
function at the cost of sacrificing the ability to calculate estimators
for any possible sample. Another avenue leads to alternate estimators
based on observing a sample of the votes from a subset of the population.
A third direction we are planning to explore is the generalisation
of the Curie-Weiss model to interacting groups of voters. This research
programme will hopefully provide a clear picture of how to estimate
interaction parameters in voting models, with applicability beyond
the Curie-Weiss model.

In summary, this article situates the Curie-Weiss model at the intersection
of statistical physics, statistics, and political decision-making.
By developing and exhaustively analysing the maximum likelihood estimator
for the coupling parameters, we aim to provide a rigorous and interpretable
method for quantifying interaction strength in voting populations,
with applications ranging from international councils to algorithmic
aggregation in digital platforms. Given this applicability to different
academic disciplines, we have tried to provide complete and comprehensible
proofs that appeal to a wide audience composed of mathematicians,
statisticians, physicists, economists, and political scientists. The
main results of this article are Propositions \ref{prop:beta^hat_well-def}
and \ref{prop:atyp_beta^hat}, and Theorem \ref{thm:properties_bML_fin}
(to be found in Section \ref{sec:estimator_results}). Since the estimator
we study (see Definition \ref{def:estimator_fin_N}) is defined by
an implicit condition (\ref{eq:opt}), we require Proposition \ref{prop:beta^hat_well-def}
to be certain the estimator is uniquely determined for any sample
of observations. As we will show, there are realisations of the sample
which lead to estimates which do not correspond to our assumption
that the true parameters are non-negative real numbers. Proposition
\ref{prop:atyp_beta^hat} assures us that the probability of these
realisations decays exponentially to 0 as the sample size goes to
infinity. Finally, Theorem \ref{thm:properties_bML_fin} contains
the aforementioned statistical properties of the estimator: consistency,
asymptotic normality, and a large deviation principle.

We present the Curie-Weiss model in the next section. The maximum
likelihood estimator is defined in Section \ref{sec:estimator_results},
where we also state the main results of this article. Sections \ref{sec:Proof_props}
and \ref{sec:Proof_Theorem} contain the proofs of the main results.
Section \ref{sec:Standard_Error} is about the standard error of the
statistics in this article. The topic of optimal weights in two-tier
voting systems is treated in Section \ref{sec:Optimal-Weights}. Finally,
an Appendix to the article contains some auxiliary results we employ.

\section{The Curie-Weiss Model}

Let the sets $\IN_{N_{\lambda}}$, $\lambda\in\IN_{M}$, represent
$M$ groups of voters. We will denote the space of voting configurations
for this population by
\[
\Omega_{N_{1}+\cdots+N_{M}}\coloneq\left\{ -1,1\right\} ^{N_{1}+\cdots+N_{M}}.
\]
Each individual vote cast $x_{\lambda i}\in\Omega_{1}$ will be indexed
by $\lambda\in\IN_{M}$ denoting the group and $i\in\IN_{N_{\lambda}}$
the identity of the voter in question. We will refer to each $\left(x_{11},\ldots,x_{1N_{1}},\ldots,x_{M1},\ldots,x_{MN_{M}}\right)\in\Omega_{N_{1}+\cdots+N_{M}}$
as a voting configuration, which consists of a complete record of
the votes cast by the entire population on a certain issue. We model
their behaviour in binary voting situations with the following voting
model:
\begin{defn}
\label{def:CWM}Let $N_{\lambda}\in\IN$ and $\beta_{\lambda}\in\IR$,
$\lambda\in\IN_{M}$. We set $\boldsymbol{N}\coloneq\left(N_{1},\ldots,N_{M}\right)$
and $\boldsymbol{\beta}\coloneq\left(\beta_{1},\ldots,\beta_{M}\right)$.
The\emph{ Curie-Weiss model} (CWM) is defined for all voting configurations
$\left(x_{11},\ldots,x_{1N_{1}},\ldots,x_{M1},\ldots,x_{MN_{M}}\right)\in\Omega_{N_{1}+\cdots+N_{M}}$
by
\begin{equation}
\IP_{\boldsymbol{\beta},\boldsymbol{N}}\left(X_{11}=x_{11},\ldots,X_{MN_{M}}=x_{MN_{M}}\right)\coloneq Z_{\boldsymbol{\beta},\boldsymbol{N}}^{-1}\,\exp\left(\frac{1}{2}\sum_{\lambda=1}^{M}\frac{\beta_{\lambda}}{N_{\lambda}}\left(\sum_{i=1}^{N_{\lambda}}x_{\lambda i}\right)^{2}\right),\label{eq:CWM}
\end{equation}
where $Z_{\boldsymbol{\beta},\boldsymbol{N}}$ is a normalisation
constant called the partition function which depends on $\boldsymbol{\beta}$
and $\boldsymbol{N}$. The constants $\beta_{\lambda}$ are called
coupling parameters.
\end{defn}

In the physical context of the CWM as a model of ferromagnetism, where
$M=1$, there is a single coupling parameter $\beta$ which is the
inverse temperature. As such, the range of values for $\beta$ is
usually $\left[0,\infty\right)$. For technical reasons to do with
the range of the statistic employed to calculate the maximum likelihood
estimator, we will consider the more general definition given above.

However, as a model of voting, the CWM has non-negative coupling parameters
to reflect social cohesion. In the voting context, the coupling parameters
$\beta_{\lambda}$ measure the degree of influence the voters in group
$\lambda$ exert over each other, with the influence becoming stronger
the larger $\beta_{\lambda}$ is. As we see, the most probable voting
configurations are those with unanimous votes in favour of or against
the proposal. However, there are only two of these extreme configurations,
whereas there is a multitude of low probability configurations with
roughly equal numbers of votes for and against. This is the `conflict
between energy and entropy.' Which one of these pseudo forces dominates
depends on the magnitude of the coupling parameters.

The partition function $Z_{\boldsymbol{\beta},\boldsymbol{N}}$ is
defined by
\begin{equation}
Z_{\boldsymbol{\beta},\boldsymbol{N}}=\sum_{x\in\Omega_{N_{1}+\cdots+N_{M}}}\exp\left(\frac{1}{2}\sum_{\lambda=1}^{M}\frac{\beta_{\lambda}}{N_{\lambda}}\left(\sum_{i=1}^{N_{\lambda}}x_{\lambda i}\right)^{2}\right).\label{eq:part_fn}
\end{equation}
We set
\begin{equation}
S_{\lambda}\coloneq\sum_{i=1}^{N_{\lambda}}X_{\lambda i},\quad\lambda\in\IN_{M}.\label{eq:S_lambda}
\end{equation}
The key to understanding the behaviour of the CWM is the random vector
\[
\left(S_{1},\ldots,S_{M}\right),
\]
which represents the voting margins, i.e.\! the difference between
the numbers of yes and no votes, in each group.

\begin{notation}Throughout this article, we will use the symbol $\IE X$
for the expectation and $\IV X$ for the variance of some random variable
$X$. Capital letters such as $X$ will denote random variables, while
lower case letters such as $x$ will denote realisations of the corresponding
random variable.\end{notation}

\section{\label{sec:estimator_results}Maximum Likelihood Estimation of $\boldsymbol{\beta}$}

\subsection{The Maximum Likelihood Estimator $\hat{\boldsymbol{\beta}}_{\boldsymbol{N}}$}

We will denote by $n\in\IN$ the size of a sample of observations.
Then each sample takes values in the space
\[
\Omega_{N_{1}+\cdots+N_{M}}^{n}\coloneq\prod_{i=1}^{n}\Omega_{N_{1}+\cdots+N_{M}}.
\]
We assume that we have access to a sample of voting configurations
$\left(x^{(1)},\ldots,x^{(n)}\right)\in\Omega_{N_{1}+\cdots+N_{M}}^{n}$
composed of $n\in\IN$ i.i.d. realisations of $\left(X_{11},\ldots,X_{MN_{\lambda}}\right)$
from the CWM. The density function for such a sample is given by the
$n$-fold product of (\ref{eq:CWM}):
\begin{equation}
f\left(x^{(1)},\ldots,x^{(n)};\boldsymbol{\beta}\right)\coloneq Z_{\boldsymbol{\beta},\boldsymbol{N}}^{-n}\,\prod_{t=1}^{n}\exp\left(\frac{1}{2}\sum_{\lambda=1}^{M}\frac{\beta_{\lambda}}{N_{\lambda}}\left(\sum_{i=1}^{N_{\lambda}}x_{\lambda i}^{\left(t\right)}\right)^{2}\right).\label{eq:density_fn}
\end{equation}
For fixed $\left(x^{(1)},\ldots,x^{(n)}\right)$, the function $\boldsymbol{\beta}\in\IR^{M}\mapsto f\left(x^{(1)},\ldots,x^{(n)};\boldsymbol{\beta}\right)\in\IR$
is called the likelihood function, and
\begin{equation}
\ln f\left(x^{(1)},\ldots,x^{(n)};\boldsymbol{\beta}\right)=-n\ln Z_{\boldsymbol{\beta},\boldsymbol{N}}+\frac{1}{2}\sum_{t=1}^{n}\sum_{\lambda=1}^{M}\frac{\beta_{\lambda}}{N_{\lambda}}\left(\sum_{i=1}^{N_{\lambda}}x_{\lambda i}^{\left(t\right)}\right)^{2}\label{eq:_log_like}
\end{equation}
is the log-likelihood function.

The maximum likelihood estimator $\hat{\boldsymbol{\beta}}_{ML}$
of $\boldsymbol{\beta}$ given the sample $\left(x^{(1)},\ldots,x^{(n)}\right)$
is the value which maximises the likelihood function, i.e.\!
\[
\hat{\boldsymbol{\beta}}_{ML}\coloneq\underset{\beta'}{\arg\max}\;f\left(x^{(1)},\ldots,x^{(n)};\boldsymbol{\beta}'\right).
\]
Since $x\mapsto\ln x$ is a strictly increasing function, we can instead
identify $\hat{\boldsymbol{\beta}}_{ML}$ as the value which maximises
the log-likelihood function
\[
\hat{\boldsymbol{\beta}}_{ML}=\underset{\beta'}{\arg\max}\;\ln f\left(x^{(1)},\ldots,x^{(n)};\boldsymbol{\beta}'\right).
\]
To find the maximum of the log-likelihood function, we derive with
respect to each $\beta_{\lambda}$ and equate to 0:
\begin{align}
\frac{\textup{d}\ln f\left(x^{(1)},\ldots,x^{(n)};\boldsymbol{\beta}\right)}{\textup{d}\beta_{\lambda}} & =-\frac{n}{Z_{\boldsymbol{\beta},\boldsymbol{N}}}\frac{\textup{d}Z_{\boldsymbol{\beta},\boldsymbol{N}}}{\textup{d}\beta_{\lambda}}+\frac{1}{2N_{\lambda}}\sum_{t=1}^{n}\left(\sum_{i=1}^{N_{\lambda}}x_{\lambda i}^{\left(t\right)}\right)^{2}\overset{!}{=}0.\label{eq:FOC}
\end{align}
We continue with our calculation of the maximum likelihood estimator.
The squared sums $S_{\lambda}^{2}$ defined in (\ref{eq:S_lambda})
have expectation
\begin{align}
\IE_{\boldsymbol{\beta},\boldsymbol{N}}S_{\lambda}^{2} & =\frac{\textup{d}Z_{\boldsymbol{\beta},\boldsymbol{N}}}{\textup{d}\beta_{\lambda}}\cdot2N_{\lambda}\,Z_{\boldsymbol{\beta},\boldsymbol{N}}^{-1}\label{eq:S2_Z}
\end{align}
by Lemma \ref{lem:Z_deriv}. We substitute (\ref{eq:S2_Z}) into (\ref{eq:FOC}):
\[
\frac{\textup{d}\ln f\left(x^{(1)},\ldots,x^{(n)};\boldsymbol{\beta}\right)}{\textup{d}\beta_{\lambda}}=-\frac{n}{Z_{\boldsymbol{\beta},\boldsymbol{N}}}\frac{1}{2N_{\lambda}}Z_{\boldsymbol{\beta},\boldsymbol{N}}\IE_{\boldsymbol{\beta},\boldsymbol{N}}S_{\lambda}^{2}+\frac{1}{2N_{\lambda}}\sum_{t=1}^{n}\left(\sum_{i=1}^{N}x_{i}^{(t)}\right)^{2}.
\]
Then the optimality condition $\left.\frac{\textup{d}\ln f\left(x^{(1)},\ldots,x^{(n)};\boldsymbol{\beta}\right)}{\textup{d}\beta}\right|_{\beta=\hat{\beta}_{ML}}=0$
is equivalent to
\begin{equation}
\IE_{\hat{\boldsymbol{\beta}}_{ML},N}S_{\lambda}^{2}=\frac{1}{n}\sum_{t=1}^{n}\left(\sum_{i=1}^{N}x_{i}^{(t)}\right)^{2}.\label{eq:opt}
\end{equation}

\begin{defn}
\label{def:T_stat}We define the statistic $\,\boldsymbol{T}:\Omega_{N_{1}+\cdots+N_{M}}^{n}\rightarrow\IR$
for any realisation of the sample $x=\left(x^{(1)},\ldots,x^{(n)}\right)\in\Omega_{N_{1}+\cdots+N_{M}}^{n}$
by
\[
\boldsymbol{T}\left(x\right)\coloneq\frac{1}{n}\sum_{t=1}^{n}\left(\left(\sum_{i=1}^{N_{1}}x_{1i}^{(t)}\right)^{2},\ldots,\left(\sum_{i=1}^{N_{M}}x_{Mi}^{(t)}\right)^{2}\right).
\]
\end{defn}

\begin{rem}
$\boldsymbol{T}$ is a random vector on the probability space $\Omega_{N_{1}+\cdots+N_{M}}^{n}$
with the power set of $\Omega_{N_{1}+\cdots+N_{M}}^{n}$ as the $\sigma$-algebra
and the probability measure defined by the density function (\ref{eq:density_fn}).
We will write $\boldsymbol{T}$ for this random variable and $\boldsymbol{T}\left(x\right)$
for its realisation given a sample $x\in\Omega_{N_{1}+\cdots+N_{M}}^{n}$.
\end{rem}

\begin{prop}
\label{prop:suff}$\boldsymbol{T}$ is a sufficient statistic for
$\boldsymbol{\beta}$.
\end{prop}

We will prove this proposition in Section \ref{sec:Proof_props}.

\begin{notation}\label{notation:infty}We will write $\left[-\infty,\infty\right]$
for the compactification $\IR\cup\left\{ -\infty,\infty\right\} $
and $\left[0,\infty\right]$ for $\left[0,\infty\right)\cup\left\{ \infty\right\} $.\end{notation}
\begin{defn}[Maximum Likelihood Estimator]
\label{def:estimator_fin_N}Let $\boldsymbol{N}\in\IN^{M}$ and $\boldsymbol{\beta}\in\IR^{M}$.
The \emph{maximum likelihood estimator} of the parameter $\boldsymbol{\beta}$
is given by $\hat{\boldsymbol{\beta}}_{\boldsymbol{N}}:\Omega_{N_{1}+\cdots+N_{M}}^{n}\rightarrow\left[-\infty,\infty\right]^{M}$
such that the optimality condition (\ref{eq:opt}) holds:
\[
\IE_{\hat{\boldsymbol{\beta}}_{\boldsymbol{N}}\left(x\right),\boldsymbol{N}}\,\left(S_{1}^{2},\ldots,S_{M}^{2}\right)=\boldsymbol{T}\left(x\right),\quad x\in\Omega_{N_{1}+\cdots+N_{M}}^{n}.
\]
\end{defn}

\subsection{Main Results of the Article}

In the remainder of this section, we will state the main results concerning
the estimator $\hat{\boldsymbol{\beta}}_{\boldsymbol{N}}$. Symbols
such as `$\xrightarrow[n\rightarrow\infty]{\textup{p}}$' and `$\mathcal{N}\left(0,\sigma^{2}\right)$'
are used in a standard way (cf. Notation \ref{not:conv}).
\begin{prop}
\label{prop:beta^hat_well-def}Let $\boldsymbol{N}\in\IN^{M}$ and
$n\in\IN$. For each sample $x\in\Omega_{N_{1}+\cdots+N_{M}}^{n}$,
there is a unique value $y\in\left[-\infty,\infty\right]^{M}$ such
that $\IE_{y,\boldsymbol{N}}\left(S_{1}^{2},\ldots,S_{M}^{2}\right)=\boldsymbol{T}\left(x\right)$
holds.
\end{prop}

Therefore, the estimator $\hat{\boldsymbol{\beta}}_{\boldsymbol{N}}$
is uniquely determined for any realisation $x\in\Omega_{N_{1}+\cdots+N_{M}}^{n}$
of the sample.
\begin{prop}
\label{prop:atyp_beta^hat}Let $\boldsymbol{N}\in\IN^{M}$. For each
value of the coupling constants $\boldsymbol{\beta}>0$ \footnote{For all $x\in\IR^{M}$, $x>0$ is to be interpreted coordinate-wise,
i.e.\! $x_{i}>0,i\in\IN_{M}$.}, there is a constant $\bar{\delta}>0$ such that
\[
\IP\left\{ \hat{\boldsymbol{\beta}}_{\boldsymbol{N}}\notin\left[0,\infty\right)\right\} \leq2^{M}\exp\left(-\bar{\delta}n\right)
\]
holds for all $n\in\IN$.
\end{prop}

In (\ref{eq:delta_bar}), we will state the value of the constant
$\bar{\delta}$ in the proposition above. Proposition \ref{prop:atyp_beta^hat}
says that although the estimator $\hat{\boldsymbol{\beta}}_{\boldsymbol{N}}$
can assume negative and even infinite values, these realisations are
exponentially unlikely as the sample size increases given our assumption
on the true value of $\boldsymbol{\beta}$ meant to reflect social
cohesion.

Finally, we state a theorem about the statistical properties of the
estimator $\hat{\boldsymbol{\beta}}_{\boldsymbol{N}}$.
\begin{thm}
\label{thm:properties_bML_fin}Fix $\boldsymbol{N}\in\IN^{M}$ and
$\boldsymbol{\beta}>0$. The estimator $\hat{\boldsymbol{\beta}}_{\boldsymbol{N}}$
has the following properties:
\begin{enumerate}
\item $\hat{\boldsymbol{\beta}}_{\boldsymbol{N}}$ is consistent: $\hat{\boldsymbol{\beta}}_{\boldsymbol{N}}\xrightarrow[n\rightarrow\infty]{\textup{p}}\boldsymbol{\beta}$.
\item $\hat{\boldsymbol{\beta}}_{\boldsymbol{N}}$ is asymptotically normal:
$\sqrt{n}\left(\hat{\boldsymbol{\beta}}_{\boldsymbol{N}}-\boldsymbol{\beta}\right)\xrightarrow[n\rightarrow\infty]{\textup{d}}\mathcal{N}\left(0,\Sigma\right)$,
where the covariance matrix $\Sigma$ is diagonal with entries $\Sigma_{\lambda\lambda}=\frac{4N_{\lambda}^{2}}{\IV_{\beta_{\lambda},N_{\lambda}}S_{\lambda}^{2}},\lambda\in\IN_{M}$.
\item $\hat{\boldsymbol{\beta}}_{\boldsymbol{N}}$ satisfies a large deviations
principle with rate $n$ and rate function $\boldsymbol{J}$ defined
in (\ref{eq:bold_J}). $\boldsymbol{J}$ has a unique minimum at $\boldsymbol{\beta}$,
and we have for each closed set $K\subset\left[-\infty,\infty\right]^{M}$
that does not contain $\boldsymbol{\beta}$, $\inf_{y\in K}\boldsymbol{J}\left(y\right)>0$
and
\[
\IP\left\{ \hat{\boldsymbol{\beta}}_{\boldsymbol{N}}\in K\right\} \leq2^{M}\exp\left(-n\inf_{y\in K}\boldsymbol{J}\left(y\right)\right)
\]
for all $n\in\IN$.
\end{enumerate}
\end{thm}

One could jump to the conclusion that given these results, the estimation
problem of the parameter $\boldsymbol{\beta}$ using the maximum likelihood
estimator from Definition \ref{def:estimator_fin_N} is solved in
satisfactory fashion. However, there are computational problems with
the calculation of the expectation $\IE_{\hat{\boldsymbol{\beta}}_{\boldsymbol{N}}\left(x\right),\boldsymbol{N}}\,\left(S_{1}^{2},\ldots,S_{M}^{2}\right)$
in (\ref{eq:opt}) for all but small $N_{1},\ldots,N_{M}$. The main
difficulty is the calculation of the normalisation constant $Z_{\boldsymbol{\beta},\boldsymbol{N}}$
in (\ref{eq:part_fn}) which is of order $2^{N_{1}+\cdots+N_{M}}$
as $N_{1}+\cdots+N_{M}\rightarrow\infty$. Two possible solutions
to this problem are:
\begin{enumerate}
\item Find and use an asymptotic approximation of $\IE_{\hat{\boldsymbol{\beta}}_{\boldsymbol{N}}\left(x\right),\boldsymbol{N}}\,\left(S_{1}^{2},\ldots,S_{M}^{2}\right)$
valid for large $N_{1}+\cdots+N_{M}$ which is less costly to calculate
than the exact moment $\IE_{\hat{\boldsymbol{\beta}}_{\boldsymbol{N}}\left(x\right),\boldsymbol{N}}\,\left(S_{1}^{2},\ldots,S_{M}^{2}\right)$.
\item Employ alternate estimators based on small subsets of voters so that
instead of the moment\\
$\IE_{\hat{\boldsymbol{\beta}}_{\boldsymbol{N}}\left(x\right),\boldsymbol{N}}\,\left(S_{1}^{2},\ldots,S_{M}^{2}\right)$
some other expression can be employed which is less costly to calculate.
This approach has the added benefit that we need less data to estimate
$\boldsymbol{\beta}$. Instead of requiring access to a sample of
voting configurations from the entire population, a sample containing
observations of a subset of votes suffices.
\end{enumerate}
We will explore both of these avenues in future work.

\section{\label{sec:Proof_props}Proof of Propositions \ref{prop:beta^hat_well-def}
and \ref{prop:atyp_beta^hat}}

We will analyse the properties of the estimator $\hat{\boldsymbol{\beta}}_{\boldsymbol{N}}$,
with Propositions \ref{prop:ES2_fin} and \ref{prop:atyp_T} being
the key insights for the proof of Propositions \ref{prop:beta^hat_well-def}
and \ref{prop:atyp_beta^hat}. Proposition \ref{prop:beta^hat_well-def}
will follow from Proposition \ref{prop:ES2_fin}, and Proposition
\ref{prop:atyp_beta^hat} from Propositions \ref{prop:ES2_fin} and
\ref{prop:atyp_T}.

First, we will prove Proposition \ref{prop:suff} about the sufficiency
of the statistic $\boldsymbol{T}$:
\begin{proof}
[Proof of Proposition \ref{prop:suff}]We observe that the density
function of the sample distribution given in (\ref{eq:density_fn})
only depends on the observations through the realisation of the statistic
$\boldsymbol{T}$:
\begin{align*}
f\left(x^{(1)},\ldots,x^{(n)};\boldsymbol{\beta}\right) & =Z_{\boldsymbol{\beta},\boldsymbol{N}}^{-n}\,\prod_{t=1}^{n}\exp\left(\frac{1}{2}\sum_{\lambda=1}^{M}\frac{\beta_{\lambda}}{N_{\lambda}}\left(\sum_{i=1}^{N_{\lambda}}x_{\lambda i}^{\left(t\right)}\right)^{2}\right)\\
 & =Z_{\boldsymbol{\beta},\boldsymbol{N}}^{-n}\,\exp\left(\frac{n}{2}\boldsymbol{T}\left(x^{(1)},\ldots,x^{(n)}\right)\left(\begin{array}{c}
\frac{\beta_{1}}{N_{1}}\\
\vdots\\
\frac{\beta_{M}}{N_{M}}
\end{array}\right)\right)\eqcolon g\left(T\left(x^{(1)},\ldots,x^{(n)}\right);\boldsymbol{\beta}\right).
\end{align*}
\end{proof}
Due to the product structure of the CWM measure from Definition \ref{def:CWM},
which features non-interacting groups, we can estimate the coupling
constants $\beta_{\lambda}$ independently of each other. We will
therefore work with the marginal distributions of each group $\lambda$
given by
\[
\IP_{\beta_{\lambda},N_{\lambda}}\left(X_{\lambda1}=x_{\lambda1},\ldots,X_{\lambda N_{\lambda}}=x_{\lambda N_{\lambda}}\right)\coloneq Z_{\beta_{\lambda},N_{\lambda}}^{-1}\,\exp\left(\frac{1}{2}\sum_{\lambda=1}^{M}\frac{\beta_{\lambda}}{N_{\lambda}}\left(\sum_{i=1}^{N_{\lambda}}x_{\lambda i}\right)^{2}\right)
\]
for all $\left(x_{\lambda1},\ldots,x_{\lambda N_{\lambda}}\right)\in\Omega_{N_{\lambda}}$,
where the partition function is
\[
Z_{\beta_{\lambda},N_{\lambda}}=\sum_{x_{\lambda}\in\Omega_{N_{\lambda}}}\exp\left(\frac{\beta_{\lambda}}{2N_{\lambda}}\left(\sum_{i=1}^{N_{\lambda}}x_{\lambda i}\right)^{2}\right).
\]
We will return to the multi-group setting in Section \ref{sec:Optimal-Weights}.
In the meantime, we will be working with a single group at a time,
and therefore there will be no confusion if we omit the subindex $\lambda$
from all our expressions to improve readability. As such, we will
be writing $\IP_{\beta,N}$ instead of $\IP_{\beta_{\lambda},N_{\lambda}}$,
$T$ instead of $\left(\boldsymbol{T}\right)_{\lambda}$, etc.
\begin{defn}
\label{def:Rademacher}The \emph{Rademacher distribution} with parameter
$t\in\left[-1,1\right]$ is a probability measure $P_{t}$ on $\left\{ -1,1\right\} $
given by $P_{t}\left\{ 1\right\} \coloneqq\frac{1+t}{2}$.
\end{defn}

We set $u\coloneq\left(1,\ldots,1\right)\in\Omega_{N}$ and for all
$x\in\Omega_{N}$
\begin{equation}
p_{\beta}\left(x\right)\coloneq\exp\left(\frac{\beta}{2N}\left(\sum_{i=1}^{N}x_{i}\right)^{2}\right).\label{eq:px}
\end{equation}
We will use the following auxiliary results in the proof of Proposition
\ref{prop:ES2_fin}:
\begin{lem}
The limits
\begin{align*}
\lim_{\beta\rightarrow\infty}\frac{p_{\beta}\left(x\right)}{p_{\beta}\left(u\right)} & =\begin{cases}
1 & \textup{if }\left|\sum_{i=1}^{N}x_{i}\right|=N,\\
0 & \textup{otherwise,}
\end{cases}
\end{align*}
hold for all $x\in\Omega_{N}$.
\end{lem}

\begin{proof}
First let $\left|\sum_{i=1}^{N}x_{i}\right|=N$. Then we have
\[
\lim_{\beta\rightarrow\infty}\frac{p_{\beta}\left(x\right)}{p_{\beta}\left(u\right)}=\lim_{\beta\rightarrow\infty}\exp\left[\frac{\beta}{2N}\left(\left(\sum_{i=1}^{N}x_{i}\right)^{2}-\left(\sum_{i=1}^{N}u_{i}\right)^{2}\right)\right]=\lim_{\beta\rightarrow\infty}\exp\left[\frac{\beta}{2N}\left(N^{2}-N^{2}\right)\right]=1.
\]
Now let $\left|\sum_{i=1}^{N}x_{i}\right|<N$. Then
\[
\lim_{\beta\rightarrow\infty}\frac{p_{\beta}\left(x\right)}{p_{\beta}\left(u\right)}=\lim_{\beta\rightarrow\infty}\exp\left[\frac{\beta}{2N}\left(\left(\sum_{i=1}^{N}x_{i}\right)^{2}-\left(\sum_{i=1}^{N}u_{i}\right)^{2}\right)\right]=0
\]
because $\left(\sum_{i=1}^{N}x_{i}\right)^{2}<N^{2}=\left(\sum_{i=1}^{N}u_{i}\right)^{2}$.
\end{proof}
The next statement follows directly from the lemma by noting that
$\left|\sum_{i=1}^{N}x_{i}\right|=N$ is equivalent to $x\in\left\{ -u,u\right\} $:
\begin{cor}
\label{cor:lim_beta}The limits
\begin{align*}
\lim_{\beta\rightarrow\infty}\frac{p_{\beta}\left(x\right)}{\sum_{y\in\Omega_{N}}p_{\beta}\left(y\right)} & =\begin{cases}
\frac{1}{2} & \textup{if }\left|\sum_{i=1}^{N}x_{i}\right|=N,\\
0 & \textup{otherwise},
\end{cases}
\end{align*}
hold for all $x\in\Omega_{N}$.
\end{cor}

\begin{defn}
\label{def:min}We define the minimum of the range of $S^{2}$ to
be $\kappa\coloneq\min_{x\in\Omega_{N}}\left\{ \left|\sum_{i=1}^{N}x_{i}\right|\right\} $
and the set
\[
\Upsilon\coloneq\left\{ x\in\Omega_{N}\,\left|\,\left|\sum_{i=1}^{N}x_{i}\right|=\kappa\right.\right\} .
\]
\end{defn}

\begin{rem}
Note that
\[
\kappa=\begin{cases}
0 & \textup{if }N\textup{ is even},\\
1 & \textup{otherwise},
\end{cases}
\]
and the cardinality of $\Upsilon$ is
\[
\left|\Upsilon\right|=\begin{cases}
\left(\begin{array}{c}
N\\
\frac{N}{2}
\end{array}\right) & \textup{if }N\textup{ is even},\\
\left(\begin{array}{c}
N\\
\frac{N+1}{2}
\end{array}\right) & \textup{otherwise}.
\end{cases}
\]
\end{rem}

\begin{lem}
Let $y\in\Upsilon$. Then the limits
\begin{align*}
\lim_{\beta\rightarrow-\infty}\frac{p_{\beta}\left(x\right)}{p_{\beta}\left(y\right)} & =\begin{cases}
1 & \textup{if }\left|\sum_{i=1}^{N}x_{i}\right|=\kappa,\\
0 & \textup{otherwise,}
\end{cases}
\end{align*}
hold for all $x\in\Omega_{N}$.
\end{lem}

\begin{proof}
Let $x,y\in\Upsilon$. Then we have
\[
\lim_{\beta\rightarrow-\infty}\frac{p_{\beta}\left(x\right)}{p_{\beta}\left(y\right)}=\lim_{\beta\rightarrow-\infty}\exp\left[\frac{\beta}{2N}\left(\left(\sum_{i=1}^{N}x_{i}\right)^{2}-\left(\sum_{i=1}^{N}y_{i}\right)^{2}\right)\right]=\lim_{\beta\rightarrow-\infty}\exp\left[\frac{\beta}{2N}\left(\kappa-\kappa\right)\right]=1.
\]
Now let $x\notin\Upsilon$. Then
\[
\lim_{\beta\rightarrow-\infty}\frac{p_{\beta}\left(x\right)}{p_{\beta}\left(y\right)}=\lim_{\beta\rightarrow-\infty}\exp\left[\frac{\beta}{2N}\left(\left(\sum_{i=1}^{N}x_{i}\right)^{2}-\left(\sum_{i=1}^{N}y_{i}\right)^{2}\right)\right]=0
\]
because $\left(\sum_{i=1}^{N}x_{i}\right)^{2}-\left(\sum_{i=1}^{N}y_{i}\right)^{2}>0$.
\end{proof}
The next corollary follows immediately from the last lemma:
\begin{cor}
\label{cor:lim_beta_neg}The limits
\begin{align*}
\lim_{\beta\rightarrow-\infty}\frac{p_{\beta}\left(x\right)}{\sum_{y\in\Omega_{N}}p_{\beta}\left(y\right)} & =\begin{cases}
\frac{1}{\left|\Upsilon\right|} & \textup{if }\left|\sum_{i=1}^{N}x_{i}\right|=\kappa,\\
0 & \textup{otherwise},
\end{cases}
\end{align*}
hold for all $x\in\Omega_{N}$.
\end{cor}

\begin{lem}
\label{lem:ES2_outside}The following statements hold:
\begin{enumerate}
\item $\lim_{\beta\rightarrow-\infty}\IE_{\beta,N}S^{2}=\kappa$.
\item $\lim_{\beta\rightarrow\infty}\IE_{\beta,N}S^{2}=N^{2}$.
\end{enumerate}
\end{lem}

We will write
\[
s\left(x\right)\coloneq\sum_{i=1}^{N}x_{i}
\]
for all $x\in\Omega_{N}$.
\begin{proof}
We show each statement in turn.
\begin{enumerate}
\item We start by proving the statement
\[
\lim_{\beta\rightarrow-\infty}\IE_{\beta,N}S^{2}=\kappa.
\]
The moment $\IE_{\beta,N}S^{2}$ is calculated as the sum of $2^{N}$
summands of the type
\[
Z_{\beta,N}^{-1}\,s\left(x\right)^{2}\exp\left(\frac{\beta}{2N}s\left(x\right)^{2}\right),\quad x\in\Omega_{N}.
\]
Since the number of summands is fixed over all values of $\beta\in\IR$,
we have
\[
\lim_{\beta\rightarrow-\infty}\IE_{\beta,N}S^{2}=\sum_{x\in\Omega_{N}}s\left(x\right)^{2}\lim_{\beta\rightarrow-\infty}\frac{\exp\left(\frac{\beta}{2N}s\left(x\right)^{2}\right)}{Z_{\beta,N}}=\sum_{x\in\Omega_{N}}s\left(x\right)^{2}\lim_{\beta\rightarrow-\infty}\frac{p_{\beta}\left(x\right)}{\sum_{y\in\Omega}p_{\beta}\left(y\right)},
\]
where $p_{\beta}\left(x\right)$ is defined as in (\ref{eq:px}).
By Definition \ref{def:min} of $\kappa$ and $\Upsilon$, for all
$x\in\Upsilon$,
\[
s\left(x\right)^{2}=\kappa.
\]
 According to Corollary \ref{cor:lim_beta_neg}, we obtain
\[
\sum_{x\in\Omega_{N}}s\left(x\right)^{2}\lim_{\beta\rightarrow-\infty}\frac{p_{\beta}\left(x\right)}{\sum_{y\in\Omega_{N}}p_{\beta}\left(y\right)}=\sum_{x\in\Upsilon}s\left(x\right)^{2}\frac{1}{\left|\Upsilon\right|}=\kappa.
\]
\item We next show the second statement employing Corollary \ref{cor:lim_beta},
$s\left(u\right)^{2}=s\left(-u\right)^{2}=N^{2}$, and the fact that
there are finitely many configurations $x\in\Omega_{N}$:
\begin{align*}
\lim_{\beta\rightarrow\infty}\IE_{\beta,N}S^{2} & =\lim_{\beta\rightarrow\infty}\frac{\sum_{x\in\Omega_{N}}p_{\beta}\left(x\right)s(x)^{2}}{\sum_{y\in\Omega_{N}}p_{\beta}\left(y\right)}=\lim_{\beta\rightarrow\infty}\frac{p_{\beta}\left(u\right)}{p_{\beta}\left(u\right)+p_{\beta}\left(-u\right)}s(u)^{2}+\lim_{\beta\rightarrow\infty}\frac{p_{\beta}\left(-u\right)}{p_{\beta}\left(u\right)+p_{\beta}\left(-u\right)}s(-u)^{2}\\
 & =\frac{1}{2}s(u)^{2}+\frac{1}{2}s(-u)^{2}=N^{2}.
\end{align*}
\end{enumerate}
\end{proof}
As we will use the function $\beta\in\IR\mapsto\IE_{\beta,N}S^{2}\in\IR$
in our proofs later on, it is convenient to assign a letter to this
function.
\begin{defn}
\label{def:expec_S2}Let for each $N\in\IN$ the function $\vartheta_{N}:\left[-\infty,\infty\right]\rightarrow\IR$
be defined by
\[
\vartheta_{N}\left(\beta\right)\coloneq\IE_{\beta,N}S^{2},\;\beta\in\IR,\qquad\vartheta_{N}\left(-\infty\right)\coloneq\kappa,\qquad\textup{and}\;\vartheta_{N}\left(\infty\right)\coloneq N^{2}.
\]
\end{defn}

\begin{prop}
\label{prop:ES2_fin}The function $\vartheta_{N}$ has the following
properties:
\begin{enumerate}
\item $\vartheta_{N}$ is strictly increasing and continuous; it is continuously
differentiable on $\IR$. Its derivative is
\[
\vartheta'_{N}\left(\beta\right)=\frac{1}{2N}\IV_{\beta,N}S^{2},\quad\beta\in\IR.
\]
\item We have for all $\beta\in\left(-\infty,0\right)$ and all $N\in\IN$:
\[
\kappa<\IE_{\beta,N}S^{2}<N.
\]
\item We have for all $\beta\in\left[0,\infty\right)$ and all $N\in\IN$:
\[
N\leq\IE_{\beta,N}S^{2}<N^{2}.
\]
\end{enumerate}
\end{prop}

\begin{proof}
\begin{enumerate}
\item The continuity on $\IR$ is clear. The function $\vartheta_{N}$ is
a sum of differentiable functions of $\beta\in\IR$, and we calculate
\begin{align*}
\frac{\textup{d}\IE_{\beta,N}S^{2}}{\textup{d}\beta} & =\frac{\textup{d}}{\textup{d}\beta}\left[\sum_{x\in\Omega_{N}}\left(\sum_{i=1}^{N}x_{i}\right)^{2}Z_{\beta,N}^{-1}\exp\left(\frac{\beta}{2N}\left(\sum_{i=1}^{N}x_{i}\right)^{2}\right)\right]\\
 & =\sum_{x\in\Omega_{N}}\left[\frac{\frac{1}{2N}\left(\sum_{i=1}^{N}x_{i}\right)^{4}\exp\left(\frac{\beta}{2N}\left(\sum_{i=1}^{N}x_{i}\right)^{2}\right)Z_{\beta,N}}{Z_{\beta,N}^{2}}\right.\\
 & \qquad-\left.\frac{\left(\sum_{i=1}^{N}x_{i}\right)^{2}\exp\left(\frac{\beta}{2N}\left(\sum_{i=1}^{N}x_{i}\right)^{2}\right)\frac{\textup{d}Z_{\beta,N}}{\textup{d}\beta}}{Z_{\beta,N}^{2}}\right]\\
 & =\frac{1}{2N}\IE_{\beta,N}S^{4}-Z_{\beta,N}^{-2}\sum_{x\in\Omega_{N}}\left(\sum_{i=1}^{N}x_{i}\right)^{2}\exp\left(\frac{\beta}{2N}\left(\sum_{i=1}^{N}x_{i}\right)^{2}\right)\\
 & \quad\cdot\sum_{y\in\Omega_{N}}\frac{1}{2N}\left(\sum_{i=1}^{N}y_{i}\right)^{2}\exp\left(\frac{\beta}{2N}\left(\sum_{i=1}^{N}y_{i}\right)^{2}\right)\\
 & =\frac{1}{2N}\IE_{\beta,N}S^{4}-\frac{1}{2N}\sum_{x\in\Omega_{N}}\left(\sum_{i=1}^{N}x_{i}\right)^{2}Z_{\beta,N}^{-1}\exp\left(\frac{\beta}{2N}\left(\sum_{i=1}^{N}x_{i}\right)^{2}\right)\\
 & \quad\cdot\sum_{y\in\Omega_{N}}\left(\sum_{i=1}^{N}y_{i}\right)^{2}Z_{\beta,N}^{-1}\exp\left(\frac{\beta}{2N}\left(\sum_{i=1}^{N}y_{i}\right)^{2}\right)\\
 & =\frac{1}{2N}\IE_{\beta,N}S^{4}-\frac{1}{2N}\left(\IE_{\beta,N}S^{2}\right)^{2}=\frac{1}{2N}\IV_{\beta,N}S^{2}>0,
\end{align*}
where we used the derivative of $Z_{\beta,N}$ provided in Lemma \ref{lem:Z_deriv},
and the strict inequality stems from the fact that the random variable
$S^{2}$ is not almost surely constant for any $\beta\in\IR$. Thus,
$\vartheta_{N}|\IR$ is continuously differentiable and strictly increasing.
The continuity of $\vartheta_{N}$ follows from the limit statements
in Lemma \ref{lem:ES2_outside}. Lemma \ref{lem:ES2_outside} and
the inequalities in points 2 and 3, which we will show now, yield
the strict monotonicity of $\vartheta_{N}$ on its entire domain.
\item For $\beta=0$, the $X_{i}$, $i\in\IN_{N}$, are independent Rademacher
random variables (see Definition \ref{def:Rademacher}) with parameter
$t=0$. We calculate the second moment of $S$:
\begin{align*}
\IE_{0,N}S^{2} & =\IE_{0,N}\left(\sum_{i=1}^{N}X_{i}\right)^{2}=\IE_{0,N}\left(\sum_{i=1}^{N}X_{i}^{2}+\sum_{i\neq j}X_{i}X_{j}\right)\\
 & =\sum_{i=1}^{N}\IE_{0,N}X_{i}^{2}+\sum_{i\neq j}\IE_{0,N}\left(X_{i}X_{j}\right)\\
 & =\sum_{i=1}^{N}1+\sum_{i\neq j}\IE_{0,N}X_{i}\,\IE_{0,N}X_{j}\\
 & =N,
\end{align*}
where we used $\IE_{0,N}X_{i}=0$, $X_{i}^{2}=1$ and the independence
of all $X_{i},X_{j}$ with $i\neq j$ under $\IP_{0,N}$.\\
Next, we use $S^{2}\geq\kappa$, $\IE_{0,N}S^{2}=N$, and the strict
monotonicity of the function $\vartheta_{N}|\IR$ to establish that
for all $\beta\in\left(-\infty,0\right)$,
\[
\kappa\leq\lim_{\beta'\rightarrow-\infty}\IE_{\beta',N}S^{2}<\IE_{\beta,N}S^{2}<\IE_{0,N}S^{2}=N.
\]
This proves the second statement.
\item Using the first property and $\lim_{\beta\rightarrow\infty}\IE_{\beta,N}S^{2}=N^{2}$
from Lemma \ref{lem:ES2_outside}, we have for all $N\in\IN$ and
all $\beta\geq0$,
\[
N=\IE_{0,N}S^{2}\leq\IE_{\beta,N}S^{2}<\lim_{\beta'\rightarrow\infty}\IE_{\beta',N}S^{2}=N^{2}.
\]
\end{enumerate}
\end{proof}
Recall that we convened to write $T=\left(\boldsymbol{T}\right)_{\lambda}$.
Proposition \ref{prop:ES2_fin} assures us that any realisation of
the statistic $T$ in the interval $\left[N,N^{2}\right)$ allows
us to identify a unique value of $\hat{\beta}_{N}\in\left[0,\infty\right)$
such that the optimality condition (\ref{eq:opt}) holds. It should
be noted that realisations $T\left(x\right)$ in $\left[\kappa,N\right)\cup\left\{ N^{2}\right\} $
are possible since the range of $S^{2}$ is $\left\{ \kappa,\left(\kappa+2\right)^{2},\ldots,N^{2}\right\} $.
As $T$ is defined to be the average of the realisations of $S^{2}$
over all observations $x^{\left(t\right)}$ in the sample $x\in\Omega_{N}^{n}$,
this implies that the range of $T$ includes $N^{2}$ and a subset
of $\left[\kappa,N\right)$.

Taking into account that $\IE\,T=\IE_{\beta,N}S^{2}$, we obtain the
following realisations of $\hat{\beta}_{N}$ depending on the value
of $T$:
\begin{enumerate}
\item If $T\left(x\right)\in\left(\kappa,N\right)$, then a negative value
for $\hat{\beta}_{N}$ satisfies (\ref{eq:opt}), and for $T\left(x\right)=\kappa$,
(\ref{eq:opt}) holds if $\hat{\beta}_{N}=-\infty$.
\item If $T\left(x\right)\in\left[N,N^{2}\right)$, then (\ref{eq:opt})
is satisfied for a unique value $\hat{\beta}_{N}\in\left[0,\infty\right)$.
\item If $T\left(x\right)=N^{2}$, then (\ref{eq:opt}) holds if $\hat{\beta}_{N}=\infty$
but not for any finite value of $\hat{\beta}_{N}$.
\end{enumerate}
These observations are the reason we defined $\hat{\beta}_{N}$ as
a function which takes values in the extended real numbers, including
$-\infty$ and $\infty$. This concludes the proof of Proposition
\ref{prop:beta^hat_well-def}.

We are interested in the case that the coupling constant $\beta$
lies in $\left[0,\infty\right)$, as this models social cohesion.
If $\beta>0$, the edge case $T\left(x\right)\in\left[\kappa,N\right)\cup\left\{ N^{2}\right\} $
is of little practical importance due to Proposition \ref{prop:atyp_T}
below. We will use some concepts and results in the proof of said
proposition, which we state before the proposition itself.

We will work with rate functions which are defined as Legendre transforms.
\begin{defn}
\label{def:Legendre}Let $f:\IR\rightarrow\IR$ be a convex function.
Then the Legendre transform $f^{*}:\IR\rightarrow\IR\cup\{\infty\}$
of $f$ is defined by $f^{*}(t)\coloneq\sup_{x\in\IR}\left\{ xt-f(x)\right\} ,t\in\IR$.
\end{defn}

We will employ two lemmas concerning the convexity of the Legendre
transformation of convex functions to prove Proposition \ref{prop:atyp_T}
below.
\begin{lem}
\label{lem:conv}Let $f:\IR\rightarrow\IR$ be a convex function.
Then the Legendre transform $f^{*}$ is convex.
\end{lem}

\begin{proof}
Let $\theta\in\left[0,1\right]$ and $t_{1},t_{2}\in\IR$. Then
\begin{align*}
\theta f^{*}\left(t_{1}\right)+\left(1-\theta\right)f^{*}\left(t_{2}\right) & =\theta\sup_{x\in\IR}\left\{ t_{1}x-f(x)\right\} +\left(1-\theta\right)\sup_{x\in\IR}\left\{ t_{2}x-f(x)\right\} \\
 & =\sup_{x\in\IR}\left\{ \theta t_{1}x-\theta f(x)\right\} +\sup_{x\in\IR}\left\{ \left(1-\theta\right)t_{2}x-\left(1-\theta\right)f(x)\right\} \\
 & \geq\sup_{x\in\IR}\left\{ \left(\theta t_{1}+\left(1-\theta\right)t_{2}\right)x-\left(\theta+\left(1-\theta\right)\right)f(x)\right\} \\
 & =f^{*}\left(\theta t_{1}+\left(1-\theta\right)t_{2}\right),
\end{align*}
where we used first the definition of $f^{*}$, then $\theta\in\left[0,1\right]$,
and finally $\sup_{x\in\IR}g(x)+\sup_{x\in\IR}h(x)\geq$\\
$\sup_{x\in\IR}\left\{ g(x)+h(x)\right\} $ for any functions $g,h:\IR\rightarrow\IR$.
\end{proof}
\begin{lem}
\label{lem:strict_conv}Let $f:\IR\rightarrow\IR$ be a convex and
differentiable function. Then the Legendre transform $f^{*}$ is strictly
convex on the set $\left\{ t\in\IR\,|\,f^{*}\left(t\right)<\infty\right\} $.
\end{lem}

\begin{proof}
We already know from Lemma \ref{lem:conv} that $f^{*}$ is convex.
Assume it is not strictly convex on $F\coloneq\left\{ t\in\IR\,|\,f^{*}\left(t\right)<\infty\right\} $.
Since $f$ is assumed to be differentiable, and hence continuous,
this implies the existence of $t_{1},t_{2}\in F,t_{1}\neq t_{2}$,
such that
\begin{equation}
f^{*}\left(\frac{t_{1}+t_{2}}{2}\right)=\frac{1}{2}f^{*}\left(t_{1}\right)+\frac{1}{2}f^{*}\left(t_{2}\right).\label{eq:midpoint}
\end{equation}
See \cite[p. 12]{Donoghue1969} for a reference that for continuous
real-valued functions defined on open intervals midpoint convexity
is equivalent to convexity.

Let $z$ be a maximiser of the set $\left\{ \left.x\frac{t_{1}+t_{2}}{2}-f(x)\,\right|\,x\in\IR\right\} $.
Then the inequality
\[
f^{*}\left(\frac{t_{1}+t_{2}}{2}\right)=z\frac{t_{1}+t_{2}}{2}-f(z)\geq x\frac{t_{1}+t_{2}}{2}-f(x),\quad x\in\IR,
\]
holds. By (\ref{eq:midpoint}),
\begin{align*}
\frac{1}{2}f^{*}\left(t_{1}\right)+\frac{1}{2}f^{*}\left(t_{2}\right) & =f^{*}\left(\frac{t_{1}+t_{2}}{2}\right)=z\frac{t_{1}+t_{2}}{2}-f(z)\\
 & =\frac{1}{2}\left(zt_{1}-f(z)\right)+\frac{1}{2}\left(zt_{2}-f(z)\right)\\
 & \leq\frac{1}{2}f^{*}\left(t_{1}\right)+\frac{1}{2}f^{*}\left(t_{2}\right),
\end{align*}
and hence $z$ is also a maximiser of the sets $\left\{ \left.xt_{i}-f(x)\,\right|\,x\in\IR\right\} $,
$i=1,2$, i.e.\!
\[
zt_{i}-f(z)\geq xt_{i}-f(x),\quad x\in\IR,\;i=1,2.
\]
By rearranging terms, we obtain
\[
f(x)\geq xt_{i}-zt_{i}+f(z),\quad x\in\IR,\;i=1,2.
\]
Thus we have two affine functions $x\mapsto g_{i}(x)\coloneq xt_{i}-zt_{i}+f(z),i=1,2$,
with the properties
\[
g_{1}\neq g_{2},\quad g_{i}\leq f,\quad g_{i}(z)=f(z),\quad i=1,2.
\]
Hence each $g_{i}$ is tangent to $f$ at $z$. However, due to the
assumed differentiability of $f$ and $g_{1}\neq g_{2}$ this is impossible,
yielding a contradiction. Therefore, (\ref{eq:midpoint}) cannot hold.
\end{proof}
\begin{defn}
\label{def:cumul_entropy}Let $P$ be a probability measure on $\IR$.
Then let $\Lambda_{P}\left(t\right)\coloneq\ln\,\int_{\IR}\exp\left(tx\right)P\left(\textup{d}x\right)$
for all $t\in\IR$ such that the expression is finite. We call $\Lambda_{P}$
the cumulant generating function of $P$ and the Legendre transform
$\Lambda_{P}^{*}$ its entropy function. Let $Y$ be a real random
variable with distribution $P$. We will then say that $\Lambda_{Y}\coloneq\Lambda_{P}$
and $\Lambda_{Y}^{*}\coloneq\Lambda_{P}^{*}$ are the cumulant generating
and entropy function of $Y$, respectively.
\end{defn}

As a last ingredient we will need in the proof of Proposition \ref{prop:atyp_T},
we prove some properties of $\Lambda_{S^{2}}$ and $\Lambda_{S^{2}}^{*}$.
\begin{defn}
\label{def:ess}Let $Y$ be a random variable. The value
\[
\textup{ess\,inf }Y\coloneq\sup\left\{ a\in\IR\,|\,\IP\left\{ Y<a\right\} =0\right\} 
\]
is called the essential infimum of $Y$. We convene that $\textup{ess\,inf }Y\coloneq-\infty$
if the set of essential lower bounds for $Y$ on the right hand side
of the display above is empty. The value 
\[
\textup{ess\,sup }Y\coloneq\inf\left\{ a\in\IR\,|\,\IP\left\{ Y>a\right\} =0\right\} 
\]
is called the essential supremum of $Y$. We convene that $\textup{ess\,sup }Y\coloneq\infty$
if the set on the right hand side above is empty.
\end{defn}

\begin{lem}
\label{lem:cumulant_entropy}Let $Y$ be a bounded random variable
which is not almost surely constant. The cumulant generating function
$\Lambda_{Y}$ of $Y$ and the entropy function $\Lambda_{Y}^{*}$
of $Y$ have the following properties:
\begin{enumerate}
\item $\Lambda_{Y}$ is convex and differentiable.
\item $\Lambda_{Y}^{*}$ is finite on the interval $\left(\textup{ess\,inf }Y,\textup{ess\,sup }Y\right)$
and infinite on $\left[\textup{ess\,inf }Y,\textup{ess\,sup }Y\right]^{c}$.
\item $\Lambda_{Y}^{*}$ is strictly convex on $\left(\textup{ess\,inf }Y,\textup{ess\,sup }Y\right)$.
\item $\Lambda_{Y}^{*}$ is strictly decreasing on the interval $\left(\textup{ess\,inf }Y,\IE\,Y\right)$
and strictly increasing on $\left(\IE\,Y,\textup{ess\,sup }Y\right)$.
\item $\Lambda_{Y}^{*}$ has a unique global minimum at $\IE\,Y$ with $\Lambda_{Y}^{*}\left(\IE\,Y\right)=0$.
\end{enumerate}
\end{lem}

\begin{proof}
Since $Y$ is bounded, $\Lambda_{Y}$ is well-defined and finite
for all $t\in\IR$. We now show that $\Lambda_{Y}$ is convex and
differentiable. Let $\theta\in\left[0,1\right]$ and $t_{1},t_{2}\in\IR$.
Then
\begin{align*}
\Lambda_{Y}\left(\theta t_{1}+\left(1-\theta\right)t_{2}\right) & =\ln\,\IE\exp\left(\left(\theta t_{1}+\left(1-\theta\right)t_{2}\right)Y\right)\\
 & =\ln\,\IE\left[\exp\left(t_{1}Y\right)^{\theta}\exp\left(t_{2}Y\right)^{1-\theta}\right]\\
 & \leq\ln\left[\left(\IE\exp\left(t_{1}Y\right)\right)^{\theta}\left(\IE\exp\left(t_{2}Y\right)\right)^{1-\theta}\right]\\
 & =\theta\Lambda_{Y}\left(t_{1}\right)+\left(1-\theta\right)\Lambda_{Y}\left(t_{2}\right),
\end{align*}
where we used Hölder's inequality. Note that the function $\left(t,Y\right)\mapsto g(t,Y)\coloneq\exp\left(tY\right)$
has the properties
\begin{enumerate}
\item $g\left(t,\cdot\right)$ is an integrable function with respect to
the push-forward measure $\IP\circ Y^{-1}$ for all $t\in\IR$.
\item The partial derivative $\frac{\partial g\left(t,Y\right)}{\partial t}$
exists $\IP\circ Y^{-1}$-almost surely for all $t\in\IR$.
\item $\left|\frac{\partial g\left(t,Y\right)}{\partial t}\right|=\left|Y\exp\left(tY\right)\right|$
is $\IP\circ Y^{-1}$-integrable for all $t\in\IR$.
\end{enumerate}
By Leibniz's integral rule, $\Lambda_{Y}=\ln\int g(\cdot,Y)\,\textup{d}\,\IP\circ Y^{-1}$
is differentiable with
\[
\Lambda_{Y}'(t)=\frac{\IE\left(Y\exp\left(tY\right)\right)}{\IE\exp\left(tY\right)}.
\]
Since $\Lambda_{Y}$ is convex and differentiable, $\Lambda_{Y}^{*}$
is strictly convex by Lemma \ref{lem:strict_conv}. This shows statements
1 and 3.

Let $y\in\left(\textup{ess\,inf }Y,\textup{ess\,sup }Y\right)$. We
show that $\Lambda_{Y}^{*}\left(y\right)<\infty$ holds. We have $\IP\left(Y<y\right),\IP\left(Y>y\right)>0$.
Let for all $t\in\IR$
\[
f\left(t\right)\coloneq yt-\Lambda_{Y}\left(t\right)
\]
and
\[
g\left(t\right)\coloneq\exp f\left(t\right)=\frac{\exp\left(yt\right)}{\IE\exp\left(tY\right)}.
\]
We write
\[
\IE\exp\left(tY\right)=\IE\exp\left(tY\right)\II_{\left\{ Y\leq y\right\} }+\IE\exp\left(tY\right)\II_{\left\{ y<Y\leq\textup{ess\,sup}Y\right\} }+\IE\exp\left(tY\right)\II_{\left\{ \textup{ess\,sup}Y<Y\right\} }.
\]
By Definition \ref{def:ess}, we have $\IP\left\{ \textup{ess\,sup }Y<Y\right\} =0$.
Therefore, by dividing numerator and denominator by $\exp\left(yt\right)$,
we obtain
\[
g\left(t\right)=\frac{1}{\IE\exp\left(t\left(Y-y\right)\right)\II_{\left\{ Y\leq y\right\} }+\IE\exp\left(t\left(Y-y\right)\right)\II_{\left\{ y<Y\leq\textup{ess\,sup}Y\right\} }}.
\]
We see that $\lim_{t\rightarrow\infty}\IE\exp\left(t\left(Y-y\right)\right)\II_{\left\{ Y\leq y\right\} }=0$
due to $\IP\left(Y<y\right)>0$, and\\
$\lim_{t\rightarrow\infty}\IE\exp\left(t\left(Y-y\right)\right)\II_{\left\{ y<Y\leq\textup{ess\,sup}Y\right\} }=\infty$
due to $\IP\left(Y>y\right)>0$. It follows that $\lim_{t\rightarrow\infty}g\left(t\right)=0$
and $\lim_{t\rightarrow\infty}f\left(t\right)=-\infty$.\\
Next we note that $\lim_{t\rightarrow-\infty}\IE\exp\left(t\left(Y-y\right)\right)\II_{\left\{ Y\leq y\right\} }=\infty$
due to $\IP\left(Y<y\right)>0$,\\
and $\lim_{t\rightarrow-\infty}\IE\exp\left(t\left(Y-y\right)\right)\II_{\left\{ y<Y\leq\textup{ess\,sup}Y\right\} }=0$
due to $\IP\left(Y>y\right)>0$. Therefore, $\lim_{t\rightarrow-\infty}g\left(t\right)=0$
and $\lim_{t\rightarrow-\infty}f\left(t\right)=-\infty$. Hence, the
continuous function $f$ reaches its maximum at some point $t_{0}\in\IR$,
and this implies $\Lambda_{Y}^{*}\left(y\right)=\sup_{t\in\IR}\left\{ yt-\Lambda_{Y}\left(t\right)\right\} =\max_{t\in\IR}f\left(t\right)=f\left(t_{0}\right)<\infty$.
This shows that $\Lambda_{Y}^{*}$ is finite on the interval $\left(\textup{ess\,inf }Y,\textup{ess\,sup }Y\right)$.

Now let $y\in\left[\textup{ess\,inf }Y,\textup{ess\,sup }Y\right]^{c}$.
Assume first $y>\textup{ess\,sup }Y$. As $\IP\left\{ \textup{ess\,sup }Y<Y\right\} =0$,
we can write
\[
\IE\exp\left(tY\right)=\IE\exp\left(tY\right)\II_{\left\{ Y\leq\textup{ess\,sup}Y\right\} }
\]
and
\[
g\left(t\right)=\frac{1}{\IE\exp\left(t\left(Y-y\right)\right)\II_{\left\{ Y\leq\textup{ess\,sup}Y\right\} }}.
\]
Then $\lim_{t\rightarrow\infty}\IE\exp\left(t\left(Y-y\right)\right)\II_{\left\{ Y\leq\textup{ess\,sup}Y\right\} }=0$
and $\lim_{t\rightarrow-\infty}\IE\exp\left(t\left(Y-y\right)\right)\II_{\left\{ Y\leq\textup{ess\,sup}Y\right\} }=\infty$.
Therefore, $\lim_{t\rightarrow\infty}g\left(t\right)=\infty$ and
$\lim_{t\rightarrow-\infty}g\left(t\right)=0$. Hence, $\Lambda_{Y}^{*}\left(y\right)=\sup_{t\in\IR}\left\{ yt-\Lambda_{Y}\left(t\right)\right\} =\lim_{t\rightarrow\infty}f\left(t\right)=\infty$.
Analogously, one can show $\Lambda_{Y}^{*}\left(y\right)=\sup_{t\in\IR}\left\{ yt-\Lambda_{Y}\left(t\right)\right\} =\lim_{t\rightarrow\infty}f\left(t\right)=\infty$
for all $y<\textup{ess\,inf }Y$. This concludes the proof of statement
2.

Next we show that $\Lambda_{Y}^{*}\left(\IE Y\right)=0$ and $\Lambda_{Y}^{*}(x)>0$
for all $x\neq\IE\,Y$. Jensen's inequality yields
\begin{equation}
\Lambda_{Y}\left(t\right)=\ln\,\IE\exp\left(tY\right)\geq\IE\left(\ln\exp\left(tY\right)\right)=t\,\IE\,Y.\label{eq:Lambda}
\end{equation}
It follows directly from $\Lambda(0)=0$ and Definition \ref{def:Legendre}
that $\Lambda_{Y}^{*}\left(\IE\,Y\right)=0$, and the second part
of statement 5 is proved.

We rearrange terms in (\ref{eq:Lambda}) to obtain for all $t<0$
and $x\geq\IE\,Y$
\[
xt-\Lambda_{Y}\left(t\right)\leq t\IE\,Y-\Lambda_{Y}\left(t\right)\leq0.
\]
Since $\Lambda_{Y}(0)=0$, $\Lambda_{Y}^{*}(x)\geq0$ for all $x\in\IR$
is a consequence of Definition \ref{def:Legendre}. This and the last
display yield
\[
\Lambda_{Y}^{*}(x)=\sup_{t\geq0}\left\{ xt-\Lambda_{Y}(t)\right\} 
\]
for all $x\geq\IE\,Y$. As the function $x\mapsto xt-\Lambda_{Y}(t)$
given $t\geq0$ is increasing, we have shown that $\Lambda_{Y}^{*}$
is increasing on the interval $\left(\IE\,Y,\textup{ess\,sup }Y\right)$.
Together with the strict convexity of $\Lambda_{Y}^{*}$, this implies
that $\Lambda_{Y}^{*}$ is strictly increasing on the interval $\left(\IE\,Y,\textup{ess\,sup }Y\right)$.
Analogously, one can show
\begin{equation}
\Lambda_{Y}^{*}(t)=\sup_{t\leq0}\left\{ xt-\Lambda_{Y}(t)\right\} \label{eq:Lam_star_sup}
\end{equation}
for all $x\leq\IE\,Y$ and that $\Lambda_{Y}^{*}$ is strictly decreasing
on the interval $\left(\textup{ess\,inf }Y,\IE\,Y\right)$, which
completes the proof of statement 4 and the first part of statement
5.
\end{proof}
As an application of Lemma \ref{lem:cumulant_entropy}, we obtain
the exponential convergence to 0 of the probability of the set of
atypical realisations $\left\{ T\notin\left[N,N^{2}\right)\right\} $.
\begin{prop}
\label{prop:atyp_T}For any value of the coupling constant $\beta>0$,
there is a constant $\delta\coloneq\inf\left\{ \Lambda_{S^{2}}^{*}\left(t\right)\,|\,t\notin\left[N,N^{2}\right)\right\} $\\
$>0$ such that
\[
\IP\left\{ T\notin\left[N,N^{2}\right)\right\} \leq2\exp\left(-\delta n\right)
\]
holds for all $n\in\IN$.
\end{prop}

\begin{proof}
The random variable $S^{2}$ is bounded and not almost surely constant,
so Lemma \ref{lem:cumulant_entropy} applies to $\Lambda_{S^{2}}^{*}$.

Set
\begin{equation}
\delta\coloneq\inf\left\{ \Lambda_{S^{2}}^{*}\left(t\right)\,|\,t\notin\left[N,N^{2}\right)\right\} .\label{eq:delta}
\end{equation}
Due to $N<\IE_{\beta,N}S^{2}<N^{2}$ by Proposition \ref{prop:ES2_fin}
and the strict monotonicity of $\Lambda_{S^{2}}^{*}$ on each of the
intervals $\left(\kappa,\IE_{\beta,N}S^{2}\right)$ and $\left(\IE_{\beta,N}S^{2},N^{2}\right)$
by Lemma \ref{lem:cumulant_entropy}, we conclude that $\Lambda_{S^{2}}^{*}\left(N\right)>0$
and $\Lambda_{S^{2}}^{*}\left(N^{2}\right)>0$, and hence
\[
\delta=\min\left\{ \Lambda_{S^{2}}^{*}\left(N\right),\Lambda_{S^{2}}^{*}\left(N^{2}\right)\right\} >0
\]
holds. We write
\[
\IP\left\{ T\notin\left[N,N^{2}\right)\right\} =\IP\left\{ T\in\left(-\infty,N\right)\right\} +\IP\left\{ T\in\left[N^{2},\infty\right)\right\} .
\]

An application of Markov's inequality yields for all $x\leq0$
\begin{align*}
\IP\left\{ T\in\left(-\infty,N\right)\right\}  & =\IP\left\{ T-N<0\right\} \leq\IP\left\{ \exp\left(nx\left(T-N\right)\right)\geq1\right\} \leq\IE\exp\left(nx\left(T-N\right)\right)\\
 & =\exp\left(-nxN\right)\prod_{s=1}^{n}\IE\exp\left(x\left(\sum_{i=1}^{N}X_{i}^{(s)}\right)^{2}\right)=\exp\left(-nxN\right)\left[\IE\exp\left(xS^{2}\right)\right]^{n}\\
 & =\exp\left(-nxN\right)\exp\left(n\Lambda_{S^{2}}\left(x\right)\right)=\exp\left(-n\left(xN-\Lambda_{S^{2}}\left(x\right)\right)\right).
\end{align*}
As this holds for all $x\leq0$, we use $N<\IE_{\beta,N}S^{2}$ and
(\ref{eq:Lam_star_sup}) to arrive at
\begin{equation}
\IP\left\{ T\in\left(-\infty,N\right)\right\} \leq\exp\left(-n\Lambda_{S^{2}}^{*}\left(N\right)\right).\label{eq:left_UB}
\end{equation}
Similarly, we calculate the upper bound
\begin{equation}
\IP\left\{ T\in\left[N^{2},\infty\right)\right\} \leq\exp\left(-n\Lambda_{S^{2}}^{*}\left(N^{2}\right)\right).\label{eq:right_UB}
\end{equation}
Combining (\ref{eq:left_UB}) and (\ref{eq:right_UB}) yields the
claim considering the definition of $\delta$ in (\ref{eq:delta}).
\end{proof}
\begin{rem}
\label{rem:delta_bar}We define the constant $\bar{\delta}$ from
Proposition \ref{prop:atyp_beta^hat} by setting
\begin{equation}
\bar{\delta}\coloneq\sum_{\lambda=1}^{M}\delta_{\lambda},\label{eq:delta_bar}
\end{equation}
with each of the $\delta_{\lambda}$ being of the form (\ref{eq:delta})
with the entropy function in the definition being $\Lambda_{S_{\lambda}^{2}}^{*}$.
\end{rem}

\begin{proof}
[Proof of Proposition \ref{prop:atyp_beta^hat}]The equivalences
\begin{align*}
T & \in\left(-\infty,N\right)\iff\hat{\beta}_{N}\in\left[-\infty,0\right),\\
T & =N^{2}\iff\hat{\beta}_{N}=\infty
\end{align*}
hold by Proposition \ref{prop:ES2_fin}. Proposition \ref{prop:atyp_beta^hat}
now follows from Proposition \ref{prop:atyp_T}.
\end{proof}

\section{\label{sec:Proof_Theorem}Proof of Theorem \ref{thm:properties_bML_fin}}

\begin{notation}\label{not:conv}Let $\left(Y_{n}\right)_{n\in\IN}$
be a sequence of random variables and $Y$ a random variable. We will
write $Y_{n}\xrightarrow[n\rightarrow\infty]{\textup{p}}Y$ for the
statement `$Y_{n}$ converges in probability to $Y$,' i.e.\!
\[
\IP\left\{ \left|Y_{n}-Y\right|>\varepsilon\right\} \xrightarrow[n\rightarrow\infty]{}0
\]
holds for all $\varepsilon>0$.\\
Let $P_{n}$ be the distribution of $Y_{n}$, $n\in\IN$, and $P$
the distribution of $Y$. We will write $Y_{n}\xrightarrow[n\rightarrow\infty]{\textup{d}}Y$
or $Y_{n}\xrightarrow[n\rightarrow\infty]{\textup{d}}P$ for the statement
`$Y_{n}$ converges in distribution to $Y$,' i.e.\!
\[
\int f\,\textup{d}P_{n}\xrightarrow[n\rightarrow\infty]{}\int f\,\textup{d}P
\]
for all continuous and bounded functions $f:\IR\rightarrow\IR$.\\
We will refer to a normal distribution with mean $\eta\in\IR$ and
variance $\sigma^{2}>0$ as $\mathcal{N}\left(\eta,\sigma^{2}\right)$.\\
Let for any random variable $X$ and any probability measure $P$
the expression $X\sim P$ stand for `$X$ is distributed according
to $P$.'\end{notation}

In the following, we will state some auxiliary results we will use
in the proof of Theorem \ref{thm:properties_bML_fin}.

Recall that we use the symbol $\kappa\in\left\{ 0,1\right\} $ for
the minimum value the random variable $S^{2}$ can assume.

Recall Definition \ref{def:expec_S2} of the function $\vartheta_{N}$.
\begin{lem}
\label{lem:inv_fn_expec_S2}The function $\vartheta_{N}$ from Definition
\ref{def:expec_S2} has an inverse function $\vartheta_{N}^{-1}:\left[\kappa,N^{2}\right]\rightarrow\left[-\infty,\infty\right]$
which is strictly increasing and continuously differentiable on $\left(\kappa,N^{2}\right)$.
\end{lem}

\begin{proof}
The statements follow from Proposition \ref{prop:ES2_fin} and the
inverse function theorem. In particular, the continuous differentiability
of $\vartheta_{N}^{-1}$ follows from the continuous differentiability
of $\vartheta_{N}$:
\[
\vartheta'_{N}\left(x\right)>0,\quad x\in\IR.
\]
So $\vartheta_{N}^{-1}$ is differentiable and
\[
\left(\vartheta_{N}^{-1}\right)'\left(y\right)=\frac{1}{\vartheta'_{N}\left(\vartheta_{N}^{-1}\left(y\right)\right)},\quad y\in\left(\kappa,N^{2}\right).
\]
\end{proof}
\begin{rem}
With the previous lemma, we can express the estimator $\hat{\beta}_{N}:\Omega_{N}^{n}\rightarrow\left[-\infty,\infty\right]$
as
\[
\hat{\beta}_{N}\left(x\right)=\left(\vartheta_{N}^{-1}\circ T\right)\left(x\right),\quad x\in\Omega_{N}^{n}.
\]
\end{rem}

\begin{notation}Let $\left(F,\mathcal{F}\right)$ be a measurable
space. We will write for any set $A\in\mathcal{F}$, the indicator
function of $A$ as $\II_{A}:F\rightarrow\IR$,
\[
\II_{A}\left(x\right)\coloneq\begin{cases}
1 & \textup{if }x\in A,\\
0 & \textup{if }x\in F\backslash A.
\end{cases}
\]
\end{notation}

Next we present an auxiliary lemma about the convergence in distribution
of sequences of random variables.
\begin{lem}
\label{lem:conv_restr_sequence}Let $\left(Y_{n}\right)_{n\in\IN}$
be a sequence of random variables and $\left(M_{n}\right)_{n\in\IN}$
a sequence of positive numbers such that
\[
\left|Y_{n}\right|\leq M_{n},\quad n\in\IN,
\]
is satisfied. Let $\nu$ be a probability measure on $\IR$, and assume
the convergence $Y_{n}\xrightarrow[n\rightarrow\infty]{\textup{d}}\nu$.
Finally, let $\left(B_{n}\right)_{n\in\IN}$ be a sequence of measurable
sets which satisfies
\[
\IP\left\{ Y_{n}\in B_{n}\right\} =o\left(\frac{1}{M_{n}}\right).
\]
Then we have for all $z\in\IR$,
\[
\II_{\left\{ Y_{n}\in B_{n}^{c}\right\} }Y_{n}+z\II_{\left\{ Y_{n}\in B_{n}\right\} }\xrightarrow[n\rightarrow\infty]{\textup{d}}\nu.
\]
\end{lem}

\begin{proof}
Set $A_{n}\coloneq\textup{Range}\left(Y_{n}\right)\backslash B_{n}$
and
\begin{align*}
W_{n} & \coloneq\II_{\left\{ Y_{n}\in A_{n}\right\} }Y_{n}+z\II_{\left\{ Y_{n}\in B_{n}\right\} },\\
U_{n} & \coloneq\II_{\left\{ Y_{n}\in B_{n}\right\} }Y_{n}-z\II_{\left\{ Y_{n}\in B_{n}\right\} }
\end{align*}
for all $n\in\IN$. We have
\begin{equation}
Y_{n}=W_{n}+U_{n},\label{eq:summands_W_U}
\end{equation}
and hence if we show
\[
U_{n}\xrightarrow[n\rightarrow\infty]{\textup{p}}0,
\]
then
\[
W_{n}\xrightarrow[n\rightarrow\infty]{\textup{d}}\nu
\]
follows from $Y_{n}\xrightarrow[n\rightarrow\infty]{\textup{d}}\nu$,
(\ref{eq:summands_W_U}), and Theorem \ref{thm:slutsky}.

Let $\varepsilon>0$. We calculate
\begin{align*}
\IP\left\{ \left|U_{n}\right|\geq\varepsilon\right\}  & \leq\frac{\IE\left|U_{n}\right|}{\varepsilon}=\frac{1}{\varepsilon}\int_{\left\{ Y_{n}\in B_{n}\right\} }\left|Y_{n}\right|\,\textup{d}\IP\\
 & \leq\frac{1}{\varepsilon}\,M_{n}\,\IP\left\{ Y_{n}\in B_{n}\right\} \xrightarrow[n\rightarrow\infty]{}0,
\end{align*}
where the convergence to 0 follows from the assumption $\IP\left\{ Y_{n}\in B_{n}\right\} =o\left(\frac{1}{M_{n}}\right)$.
Finally,
\[
\IP\left\{ \left|z\right|\II_{\left\{ Y_{n}\in B_{n}\right\} }\geq\varepsilon\right\} \leq\frac{\left|z\right|}{\varepsilon}\,\IP\left\{ Y_{n}\in B_{n}\right\} =\xrightarrow[n\rightarrow\infty]{}0,
\]
and thus $U_{n}\xrightarrow[n\rightarrow\infty]{\textup{p}}0$ holds.
\end{proof}
The next auxiliary results relate to large deviation principles and
contraction principles (see Definition \ref{def:LDP} and Theorem
\ref{thm:contr_princ}).
\begin{lem}
Let $\mathcal{X}$ be a metric space and $I:\mathcal{X}\rightarrow\left[0,\infty\right]$
a good rate function. Then, for any non-empty closed set $K\subset\mathcal{X}$,
there is a point $x_{K}\in K$ such that
\[
I\left(x_{K}\right)=\inf_{x\in K}I\left(x\right).
\]
\end{lem}

\begin{proof}
If $\inf_{x\in K}I\left(x\right)=\infty$, then set $x_{K}$ equal
to an arbitrary point in $K$. So assume $\inf_{x\in K}I\left(x\right)<\infty$.
Let $\left(x_{n}\right)_{n\in\IN}$ be a sequence in $K$ with $\lim_{n\rightarrow\infty}I\left(x_{n}\right)=\inf_{x\in K}I\left(x\right)$.
Then there is a constant $n_{0}\in\IN$ such that for all $n\geq n_{0}$,
\[
I\left(x_{n}\right)\leq\inf_{x\in K}I\left(x\right)+1\eqcolon\alpha<\infty.
\]
We have
\[
x_{n}\in\left\{ x\in\mathcal{X}\,|\,I\left(x\right)\leq\alpha\right\} ,\quad n\geq n_{0}.
\]
As $I$ is a good rate function, the level set $\left\{ x\in\mathcal{X}\,|\,I\left(x\right)\leq\alpha\right\} $
is compact, and therefore $\left(x_{n}\right)_{n\geq n_{0}}$ has
a subsequence $\left(x_{n_{k}}\right)_{k\in\IN}$ that converges to
some  $x_{0}\in\left\{ x\in\mathcal{X}\,|\,I\left(x\right)\leq\alpha\right\} $.
Since $x_{n_{k}}\in K$ for all $k\in\IN$, and $K$ is closed, $x_{0}$
must belong to $K$. Next we note that
\begin{align*}
\inf_{x\in K}I\left(x\right) & =\lim_{n\rightarrow\infty}I\left(x_{n}\right)=\lim_{k\rightarrow\infty}I\left(x_{n_{k}}\right)\\
 & =\liminf_{k\rightarrow\infty}I\left(x_{n_{k}}\right)\geq I\left(x_{0}\right),
\end{align*}
where the inequality is due to $x_{n_{k}}\xrightarrow[k\rightarrow\infty]{}x_{0}$
and the lower semi-continuity of $I$. As $x_{0}\in K$, we also have
$I\left(x_{0}\right)\geq\inf_{x\in K}I\left(x\right)$. Thus $I\left(x_{0}\right)=\inf_{x\in K}I\left(x\right)$
holds.
\end{proof}
\begin{lem}
\label{lem:contr_princ_min}Let $\mathcal{X}$ and $\mathcal{Y}$
be metric spaces, $I:\mathcal{X}\rightarrow\left[0,\infty\right]$
a good rate function, and $f:\mathcal{X}\rightarrow\mathcal{Y}$ a
continuous function. We define $J:\mathcal{Y}\rightarrow\left[0,\infty\right]$
by
\[
J\left(y\right)\coloneq\inf\left\{ I\left(x\right)\,|\,x\in\mathcal{X},f\left(x\right)=y\right\} ,\quad y\in\mathcal{Y}.
\]
Let $M_{I}\subset\mathcal{X}$ be the set of minima of $I$ and $M_{J}\subset\mathcal{Y}$
the set of minima of $J$. Then $f\left(M_{I}\right)=M_{J}$ holds.
In particular, if $f$ is injective and $I$ has a unique minimum
at $x_{0}$, then $J$ has a unique minimum at $f\left(x_{0}\right)$.
\end{lem}

\begin{proof}
Let $y_{0}\in f\left(M_{I}\right)$ and $x_{0}\in M_{I}$ be such
that $y_{0}=f\left(x_{0}\right)$. By definition of $J$, we have
\[
J\left(y_{0}\right)=\inf\left\{ I\left(x\right)\,|\,x\in\mathcal{X},f\left(x\right)=y_{0}\right\} \leq I\left(x_{0}\right)
\]
because of $f\left(x_{0}\right)=y_{0}$. On the other hand, we have
for all $x\in\mathcal{X}$
\[
I\left(x\right)\geq I\left(x_{0}\right)
\]
as $x_{0}\in M_{I}$, and therefore,
\[
J\left(y_{0}\right)=\inf\left\{ I\left(x\right)\,|\,x\in\mathcal{X},f\left(x\right)=y_{0}\right\} \geq I\left(x_{0}\right).
\]
So we have established $J\left(y_{0}\right)=I\left(x_{0}\right)$.
Let $y\in\mathcal{Y}$. We have
\[
J\left(y\right)=\inf\left\{ I\left(x\right)\,|\,x\in\mathcal{X},f\left(x\right)=y\right\} \geq I\left(x_{0}\right)=J\left(y_{0}\right),
\]
and thus $y_{0}\in M_{J}$.

Now let $y_{0}\in M_{J}$. If $y_{0}\notin f\left(\mathcal{X}\right)$,
then
\[
J\left(y_{0}\right)=\inf\left\{ I\left(x\right)\,|\,x\in\mathcal{X},f\left(x\right)=y_{0}\right\} =\inf\emptyset=\infty.
\]
Since $I$ is a rate function, there is some $x_{1}\in\mathcal{X}$
such that $I\left(x_{1}\right)<\infty$, and $J\left(f\left(x_{1}\right)\right)\leq I\left(x_{1}\right)<\infty=J\left(y_{0}\right)$,
contradicting the assumption $y_{0}\in M_{J}$ . Hence, $y_{0}\in f\left(\mathcal{X}\right)$
must hold, and $f^{-1}\left(\left\{ y_{0}\right\} \right)$ is a non-empty
closed subset of $\mathcal{X}$ because $f$ is continuous. We apply
the previous lemma to obtain a point $x_{0}\in f^{-1}\left(\left\{ y_{0}\right\} \right)$
with the property
\[
I\left(x_{0}\right)=\inf\left\{ I\left(x\right)\,|\,x\in f^{-1}\left(\left\{ y_{0}\right\} \right)\right\} =J\left(y_{0}\right)\leq J\left(y\right),\quad y\in\mathcal{Y}.
\]
Therefore,
\[
I\left(x_{0}\right)\leq I\left(x\right),\quad x\in f^{-1}\left(\left\{ y\right\} \right),\quad y\in\mathcal{Y}.
\]
Using the identity $\mathcal{X}=\bigcup_{y\in f\left(\mathcal{X}\right)}f^{-1}\left(\left\{ y\right\} \right)$,
we obtain
\[
I\left(x_{0}\right)\leq I\left(x\right)\quad x\in\mathcal{X},
\]
and thus $x_{0}\in M_{I}$ and $y_{0}=f\left(x_{0}\right)\in f\left(M_{I}\right)$
follow.
\end{proof}
Now are ready to begin the proof proper of Theorem \ref{thm:properties_bML_fin}.
Recall that $N$ is the number of voters, $n$ is the number of observations
in the sample, $\Lambda_{S^{2}}^{*}$ is the entropy function of $S^{2}$,
and see Definition \ref{def:LDP} of large deviation principles. Also
recall Lemma \ref{lem:inv_fn_expec_S2} which states that the function
$\vartheta_{N}^{-1}$ exists and is continuous and strictly increasing.
In preparation for statement 3 of Theorem \ref{thm:properties_bML_fin},
we define the good rate function $J:\left[-\infty,\infty\right]\rightarrow\left[0,\infty\right]$
by
\begin{equation}
J\left(y\right)\coloneq\inf\left\{ \Lambda_{S^{2}}^{*}\left(x\right)\,|\,x\in\IR,\vartheta_{N}^{-1}\left(x\right)=y\right\} ,\quad y\in\left[-\infty,\infty\right].\label{eqJ}
\end{equation}

\begin{rem}
\label{rem:bold_J}We define the multivariate good rate function $\boldsymbol{J}:\left[-\infty,\infty\right]^{M}\rightarrow\left[0,\infty\right]$
from the statement of Theorem \ref{thm:properties_bML_fin} by setting
\begin{equation}
\boldsymbol{J}\left(x\right)\coloneq\sum_{\lambda=1}^{M}J_{\lambda}\left(x_{\lambda}\right),\quad x\in\left[-\infty,\infty\right]^{M},\label{eq:bold_J}
\end{equation}
with each of the $J_{\lambda}$ being of the form (\ref{eqJ}).
\end{rem}

\begin{proof}
[Proof of Theorem \ref{thm:properties_bML_fin}]We prove each statement
in turn.
\begin{enumerate}
\item Recall that $\left(x^{(1)},\ldots,x^{(n)}\right)\in\Omega_{N}^{n}$
refers to the sample of voting configurations we observe. $x_{i}^{(t)}\in\left\{ -1,1\right\} $,
$t\in\IN_{n}$, $i\in\IN_{N}$, is the vote of individual $i$ in
the $t$-th vote. The weak law of large numbers
\begin{equation}
\frac{1}{n}\sum_{t=1}^{n}\left(\sum_{i=1}^{N}X_{i}^{(t)}\right)^{2}\xrightarrow[n\rightarrow\infty]{\textup{p}}\IE_{\beta,N}S^{2}\label{eq:WLLN}
\end{equation}
holds because $\left(\sum_{i=1}^{N}X_{i}^{(t)}\right)^{2}$ is a bounded
random variable, and thus the second moment exists. By Definition
\ref{def:estimator_fin_N}, we have
\[
\IE_{\hat{\beta}_{N},N}S^{2}=\frac{1}{n}\sum_{t=1}^{n}\left(\sum_{i=1}^{N}x_{i}^{(t)}\right)^{2}.
\]
This and (\ref{eq:WLLN}) yield
\[
\IE_{\hat{\beta}_{N},N}S^{2}\xrightarrow[n\rightarrow\infty]{\textup{p}}\IE_{\beta,N}S^{2}.
\]
As noted above the inverse function $\vartheta_{N}^{-1}$, which maps
an expectation $\IE_{\beta,N}S^{2}$ to the value $\beta$, is continuous.
$\hat{\beta}_{N}\xrightarrow[n\rightarrow\infty]{\textup{p}}\beta$
follows by Theorem \ref{thm:cont_mapping}.
\item By Definition \ref{def:T_stat}, the statistic $T$ has the form
specified in Proposition \ref{prop:conv_stat} with the function $f:\Omega_{N}\rightarrow\IR$
given by
\[
f\left(x_{1},\ldots,x_{N}\right)\coloneq\left(\sum_{i=1}^{N}x_{i}\right)^{2}
\]
for all $\left(x_{1},\ldots,x_{N}\right)\in\Omega_{N}$. Hence, we
have $\mu=\IE_{\beta,N}S^{2}=\IE\,T$ and $\sigma^{2}=\IV_{\beta,N}S^{2}$.
By Proposition \ref{prop:conv_stat},
\[
\sqrt{n}\left(T-\IE_{\beta,N}S^{2}\right)\xrightarrow[n\rightarrow\infty]{\textup{d}}\mathcal{N}\left(0,\IV_{\beta,N}S^{2}\right)
\]
holds. We want to apply Theorem \ref{thm:delta_method} to the transformation
$f\coloneq\vartheta_{N}^{-1}$, but face the difficulty that $\vartheta_{N}^{-1}$
is not differentiable on the entire range of $T$. By assumption,
$0<\beta<\infty$. Therefore, by Proposition \ref{prop:ES2_fin},
we have $\kappa<\IE_{\beta,N}S^{2}<N^{2}$. Let $a<b$ be real numbers
such that
\[
\kappa<a<\IE_{\beta,N}S^{2}<b<N^{2}.
\]
Set $K\coloneq\left(a,b\right)^{c}$ and $B_{n}\coloneq\left\{ \sqrt{n}\left(T-\IE_{\beta,N}S^{2}\right)\,|\,T\in K\right\} $
for all $n\in\IN$. We define
\begin{align*}
W_{n} & \coloneq\sqrt{n}\left(T-\IE_{\beta,N}S^{2}\right)\,\II_{\left\{ T\in K^{c}\right\} },\quad n\in\IN
\end{align*}
and apply Lemma \ref{lem:conv_restr_sequence} to the sequence $Y_{n}\coloneq\sqrt{n}\left(T-\IE_{\beta,N}S^{2}\right)$.\\
Note that $B_{n}$ is finite for all $n\in\IN$, so $B\coloneq\bigcup_{n=1}^{\infty}B_{n}$
is countable hence closed. Set $\nu\coloneq\mathcal{N}\left(0,\IV_{\beta,N}S^{2}\right)$
and $M_{n}\coloneq\sqrt{n}$ for each $n\in\IN$. The set $K$ is
closed and $T\in K$ if and only if $Y_{n}\in B_{n}$. By statement
4 of Proposition \ref{prop:conv_stat}, we have
\[
\IP\left\{ Y_{n}\in B\right\} \leq2\exp\left(-n\inf_{x\in K}\Lambda_{S^{2}}^{*}\left(x\right)\right),\quad n\in\IN.
\]
Therefore, $\IP\left\{ Y_{n}\in B\right\} =o\left(\frac{1}{M_{n}}\right)$
holds, and we can apply Lemma \ref{lem:conv_restr_sequence} to conclude
\[
W_{n}\xrightarrow[n\rightarrow\infty]{\textup{d}}\mathcal{N}\left(0,\IV_{\beta,N}S^{2}\right).
\]
We apply Theorem \ref{thm:delta_method}. Set $D\coloneq\left[a,b\right]$
and $f:D\rightarrow\IR$
\[
f\left(y\right)\coloneq\vartheta_{N}^{-1}\left(y\right),\quad y\in D.
\]
$f$ is continuously differentiable and the derivative $f'$ is strictly
positive on the compact set $D$. Therefore,
\begin{align*}
 & \quad\sqrt{n}\left(\hat{\beta}_{N}-\beta\right)\,\II_{\left\{ T\in K^{c}\right\} }\\
 & =\sqrt{n}\left(\hat{\beta}_{N}-\beta\right)\,\II_{\left\{ T\in K^{c}\right\} }+\sqrt{n}\left(\beta-\beta\right)\,\II_{\left\{ T\in K\right\} }\\
 & =\sqrt{n}\left(\left(\hat{\beta}_{N}\,\II_{\left\{ T\in K^{c}\right\} }+\beta\,\II_{\left\{ T\in K\right\} }\right)-\beta\right)\\
 & =\sqrt{n}\left(f\left(T\,\II_{\left\{ T\in K^{c}\right\} }+\IE_{\beta,N}S^{2}\,\II_{\left\{ T\in K\right\} }\right)-f\left(\IE_{\beta,N}S^{2}\right)\right)\xrightarrow[n\rightarrow\infty]{\textup{d}}\mathcal{N}\left(0,\left(f'\left(\mu\right)\right)^{2}\sigma^{2}\right)
\end{align*}
holds. We apply Lemma \ref{lem:conv_restr_sequence} once more to
conclude that
\[
\sqrt{n}\left(\hat{\beta}_{N}-\beta\right)\xrightarrow[n\rightarrow\infty]{\textup{d}}\mathcal{N}\left(0,\left(f'\left(\mu\right)\right)^{2}\sigma^{2}\right)
\]
is satisfied.\\
We use Proposition \ref{prop:ES2_fin} which states that $\vartheta'_{N}\left(\beta\right)=\frac{1}{2N}\IV_{\beta,N}S^{2}$,
and Lemma \ref{lem:inv_fn_expec_S2} provides the derivative
\[
\left(\vartheta_{N}^{-1}\right)'\left(\IE_{\beta,N}S^{2}\right)=\frac{1}{\vartheta'_{N}\left(\vartheta_{N}^{-1}\left(\IE_{\beta,N}S^{2}\right)\right)}=\frac{1}{\vartheta'_{N}\left(\beta\right)}=\frac{2N}{\IV_{\beta,N}S^{2}}.
\]
Substituting the value of the derivative in the previous display yields
the claim concerning the asymptotic normality of the estimator $\hat{\beta}_{N}$.
\item By Definition \ref{def:T_stat}, the statistic $T$ has the form specified
in Proposition \ref{prop:conv_stat} with the function $f:\Omega_{N}\rightarrow\IR$
given by
\[
f\left(x_{1},\ldots,x_{N}\right)\coloneq\left(\sum_{i=1}^{N}x_{i}\right)^{2}
\]
for all $\left(x_{1},\ldots,x_{N}\right)\in\Omega_{N}$. Hence, we
have $\mu=\IE_{\beta,N}S^{2}$, $\sigma^{2}=\IV_{\beta,N}S^{2}$,
and $\Lambda_{S^{2}}^{*}$ for the objects in Proposition \ref{prop:conv_stat},
and, therefore, $T$ satisfies a large deviations principle with rate
$n$ and rate function $\Lambda_{S^{2}}^{*}$. As noted before, the
function $\vartheta_{N}^{-1}$ is continuous, we have $T\left(x^{(1)},\ldots,x^{(n)}\right)\in\left[\kappa,N^{2}\right]$
for all $n\in\IN$ and all $\left(x^{(1)},\ldots,x^{(n)}\right)\in\Omega_{N}^{n}$,
and $\left[\kappa,N^{2}\right]$ is the domain of $\vartheta_{N}^{-1}$.
By Theorem \ref{thm:contr_princ}, $\hat{\beta}_{N}=\vartheta_{N}^{-1}\circ T$
satisfies a large deviations principle with rate $n$ and rate function
$J$.\\
By Lemma \ref{lem:cumulant_entropy}, $\Lambda_{S^{2}}^{*}$ has a
unique minimum at $\IE_{\beta,N}S^{2}$. Since Lemma \ref{lem:inv_fn_expec_S2}
also says that $\vartheta_{N}^{-1}$ is strictly increasing, by Lemma
\ref{lem:contr_princ_min}, $J$ has a unique minimum at $\beta=\vartheta_{N}^{-1}\left(\IE_{\beta,N}S^{2}\right)$.
Let $K\subset\IR$ be a closed set that does not contain $\beta$.
We have
\[
\IP\left\{ \hat{\beta}_{N}\in K\right\} =\IP\left\{ \vartheta_{N}\left(\hat{\beta}_{N}\right)\in\vartheta_{N}\left(K\right)\right\} =\IP\left\{ \IE_{\hat{\beta}_{N},N}S^{2}\in\vartheta_{N}\left(K\right)\right\} =\IP\left\{ T\in\vartheta_{N}\left(K\right)\right\} ,
\]
where the last step uses Definition \ref{def:estimator_fin_N}. Now
note that $K$ is a closed subset of the compact space $\left[-\infty,\infty\right]$
and therefore compact. The continuous function $\vartheta_{N}$ maps
it to the compact set $\vartheta_{N}\left(K\right)\subset\IR$. Hence,
$\vartheta_{N}\left(K\right)$ is a closed set. Since $K$ does not
contain $\beta$ and $\vartheta_{N}$ is strictly increasing, $\vartheta_{N}\left(K\right)$
does not contain $\IE_{\beta,N}S^{2}$. By Proposition \ref{prop:conv_stat},
we have
\[
\IP\left\{ T\in\vartheta_{N}\left(K\right)\right\} \leq2\exp\left(-\delta'n\right)
\]
 for all $n\in\IN$ with $\delta'\coloneq\inf_{x\in\vartheta_{N}\left(K\right)}\Lambda_{S^{2}}^{*}\left(x\right)>0$.
The equality
\begin{align*}
\delta=\inf_{y\in K}J\left(y\right) & =\inf_{y\in K}\inf\left\{ \Lambda_{S^{2}}^{*}\left(x\right)\,|\,x\in\IR,\vartheta_{N}^{-1}\left(x\right)=y\right\} =\inf\left\{ \Lambda_{S^{2}}^{*}\left(x\right)\,|\,x\in\IR,\vartheta_{N}^{-1}\left(x\right)\in K\right\} \\
 & =\inf\left\{ \Lambda_{S^{2}}^{*}\left(x\right)\,|\,x\in\vartheta_{N}\left(K\right)\right\} =\delta'
\end{align*}
yields the final claim.
\end{enumerate}
\end{proof}

\section{\label{sec:Standard_Error}Standard Error of the Statistic $T$ and
the Estimator $\hat{\beta}_{N}$}

In order to calculate moments in the proof of Lemma \ref{lem:conv_mom}
below, we introduce the concept of profile vectors.
\begin{defn}
\label{def:ind_prof}Let $k,n\in\IN$ with $k\leq n$. We will call
all $\underbar{\ensuremath{\boldsymbol{i}}}\in\IN_{n}^{k}$ \emph{index
vectors} and set
\[
\Pi\coloneq\left\{ \left(r_{1},\ldots,r_{k}\right)\in\left(\IN_{k}\right)^{k}\,\left|\,\sum_{\ell=1}^{k}\ell r_{\ell}=k\right.\right\} ,
\]
and we will refer to $\Pi$ as the set of profile vectors and to the
elements $\underbar{\ensuremath{\boldsymbol{r}}}\in\Pi$ as \emph{profile
vectors}. For any index vector $\underbar{\ensuremath{\boldsymbol{i}}}\in\IN_{n}^{k}$,
the expression $\underbar{\ensuremath{\boldsymbol{r}}}\coloneq\left(r_{1},\ldots,r_{k}\right)\coloneq\underbar{\ensuremath{\boldsymbol{\rho}}}\left(\underbar{\ensuremath{\boldsymbol{i}}}\right)$
is defined as follows: for each $\ell\in\IN_{k}$, $r_{\ell}$ is
the number of indices in $\underbar{\ensuremath{\boldsymbol{i}}}$
that appear exactly $\ell$ times. We will call $\underbar{\ensuremath{\boldsymbol{r}}}$
the profile vector of $\underbar{\ensuremath{\boldsymbol{i}}}$.
\end{defn}

\begin{rem}
It is an elementary fact that for each index vector $\underbar{\ensuremath{\boldsymbol{i}}}\in\IN_{n}^{k}$,
$\underbar{\ensuremath{\boldsymbol{\rho}}}\left(\underbar{\ensuremath{\boldsymbol{i}}}\right)\in\Pi$
holds.
\end{rem}

\begin{lem}
\label{lem:multiplicity}Let $k\in\IN$ and $\underbar{\ensuremath{\boldsymbol{r}}}\in\Pi$.
The number of index vectors $\underbar{\ensuremath{\boldsymbol{i}}}\in\IN_{n}^{k}$
such that $\underbar{\ensuremath{\boldsymbol{r}}}=\left(r_{1},\ldots,r_{k}\right)=\underbar{\ensuremath{\boldsymbol{\rho}}}\left(\underbar{\ensuremath{\boldsymbol{i}}}\right)$
is given by
\[
\frac{n!}{r_{1}!\cdots r_{k}!\left(n-\sum_{\ell=1}^{k}r_{\ell}\right)!}\frac{k!}{1!^{r_{1}}\cdots k!^{r_{k}}}.
\]
\end{lem}

\begin{proof}
We construct an index vector $\underbar{\ensuremath{\boldsymbol{i}}}$
with entries in $\IN_{n}$ and profile $\underbar{\ensuremath{\boldsymbol{r}}}=\underbar{\ensuremath{\boldsymbol{\rho}}}\left(\underbar{\ensuremath{\boldsymbol{i}}}\right)$
in two steps: 
\begin{enumerate}
\item We partition $\IN_{n}$ into $k+1$ sets $B_{j}$, $j\in\IN_{k+1}$.
Each set $B_{j}$ contains the indices which occur exactly $j$ times
in $\underbar{\ensuremath{\boldsymbol{i}}}$ for $j\leq k$ and $B_{k+1}\coloneq\IN_{n}\backslash\bigcup_{j=1}^{k}B_{j}$.
Hence, $\left|B_{j}\right|=r_{j}$ holds for all $j\leq k$. There
are
\begin{equation}
\left(\begin{array}{ccccc}
 &  &  & n!\\
r_{1}! & r_{2}! &  & r_{k}! & \left(N-\sum_{\ell=1}^{k}r_{\ell}\right)!
\end{array}\right)\label{eq:multinom}
\end{equation}
ways to realise this partition of $\IN_{n}$. 
\item The index vector $\underbar{\ensuremath{\boldsymbol{i}}}$ is of length
$k$. We can think of the elements of $\bigcup_{j=1}^{k}B_{j}$ as
a finite alphabet. Our task is to assemble an ordered $k$-tuple $\underbar{\ensuremath{\boldsymbol{i}}}$
of elements of $\bigcup_{j=1}^{k}B_{j}$ in such a fashion that for
each $j\leq k$, each of the $r_{j}$ elements of $B_{j}$ occurs
exactly $j$ times. So there are $r_{1}$ elements of $\bigcup_{j=1}^{k}B_{j}$
selected once, $r_{2}$ elements selected twice,$\ldots$, $r_{k}$
elements selected $k$ times. There are
\[
\left(\begin{array}{cccccccccc}
 &  &  &  &  & k!\\
1 & \ldots & 1 & 2 & \ldots & 2 & \ldots & k! & \ldots & k!
\end{array}\right)=\frac{k!}{1!^{r_{1}}\cdots k!^{r_{k}}}
\]
ways to accomplish this task.
\end{enumerate}
\end{proof}
The statistic $T$ from Definition \ref{def:T_stat} is an unbiased
estimator of the expectation $\IE_{\beta,N}S^{2}$. Since each summand
composing $T$ takes values in $\left[\kappa,N^{2}\right]$, $T$
is a bounded random variable.
\begin{lem}
\label{lem:conv_mom}Let $\left(Y_{n}\right)_{n\in\IN}$ be a sequence
of i.i.d. random variables with mean $\IE\,Y_{n}=\mu$, variance $\IV\,Y_{n}=\sigma^{2}$,
and existing moments $\IE\left|X_{n}\right|^{k}<\infty$ for all $k,n\in\IN$.
Let $Z$ be a random variable with distribution $\mathcal{N}\left(0,\sigma^{2}\right)$.
Then
\[
\IE\left(\frac{1}{\sqrt{n}}\sum_{i=1}^{n}\left(Y_{i}-\mu\right)\right)^{k}\xrightarrow[n\rightarrow\infty]{}\IE\,Z^{k}
\]
 holds for all $k\in\IN$.
\end{lem}

\begin{proof}
We will use the concepts from Definition \ref{def:ind_prof} and Lemma
\ref{lem:multiplicity} to calculate the expectation\\
$\IE\left(\frac{1}{\sqrt{n}}\sum_{i=1}^{n}\left(Y_{i}-\mu\right)\right)^{k}$.
To simplify the notation, we set $U_{i}\coloneq Y_{i}-\mu$ for all
$i\in\IN$. 

As the random variables $\left(U_{n}\right)_{n\in\IN}$ are i.i.d.,
we have $\IE\,U_{i_{1}}\cdots U_{i_{k}}=\IE\,U_{j_{1}}\cdots U_{j_{k}}$
for all index vectors $\underbar{\ensuremath{\boldsymbol{i}}},\underbar{\ensuremath{\boldsymbol{j}}}$
with $\underbar{\ensuremath{\boldsymbol{\rho}}}\left(\underbar{\ensuremath{\boldsymbol{i}}}\right)=\underbar{\ensuremath{\boldsymbol{\rho}}}\left(\underbar{\ensuremath{\boldsymbol{j}}}\right)$.
Therefore, we can write for any profile index vector $\underbar{\ensuremath{\boldsymbol{i}}}$
with profile vector $\underbar{\ensuremath{\boldsymbol{r}}}=\underbar{\ensuremath{\boldsymbol{\rho}}}\left(\underbar{\ensuremath{\boldsymbol{i}}}\right)$
\[
\IE\,U\left(\underbar{\ensuremath{\boldsymbol{r}}}\right)\coloneq\IE\,U_{i_{1}}\cdots X_{i_{k}}.
\]

Using Lemma \ref{lem:multiplicity}, we express the expectation $\IE\left(\frac{1}{\sqrt{n}}\sum_{i=1}^{n}U_{i}\right)^{k}$
as
\begin{align}
\IE\left(\frac{1}{\sqrt{n}}\sum_{i=1}^{n}U_{i}\right)^{k} & =\frac{1}{n^{\frac{k}{2}}}\sum_{i_{1},\ldots,i_{k}=1}^{n}\IE\,U_{i_{1}}\cdots U_{i_{k}}\nonumber \\
 & =\frac{1}{n^{\frac{k}{2}}}\sum_{\underbar{\ensuremath{\boldsymbol{r}}}\in\Pi}\frac{n!}{r_{1}!\cdots r_{k}!\left(n-\sum_{\ell=1}^{k}r_{\ell}\right)!}\frac{k!}{1!^{r_{1}}\cdots k!^{r_{k}}}\IE\,U\left(\underbar{\ensuremath{\boldsymbol{r}}}\right).\label{eq:moment_summands}
\end{align}
As a first step, we note that since $\IE\,U_{i}=0$ and $\IE\,U_{i}^{\ell}\,U_{j}^{m}=\IE\,U_{i}^{\ell}\,\IE\,U_{j}^{m}$
for all $i\neq j$ and all $\ell,m\in\IN$, we have for all $\underbar{\ensuremath{\boldsymbol{r}}}\in\Pi$
with $r_{1}\geq1$,
\begin{equation}
\IE\,U\left(\underbar{\ensuremath{\boldsymbol{r}}}\right)=0.\label{eq:single_ind}
\end{equation}

Now let $r_{1}=0$ and $r_{\ell^{*}}>0$ for some $\ell^{*}>2$. Set
$\left\lfloor x\right\rfloor \coloneq\max\left\{ m\in\IZ\,|\,m\leq x\right\} $
for all $x\in\IR$. Suppose first that $\ell^{*}=2j^{*}+1$ for some
$j^{*}\geq1$. Then
\begin{align*}
\sum_{\ell=1}^{k}r_{\ell} & =\frac{1}{2}\left[\sum_{\ell=1}^{\left\lfloor \frac{k}{2}\right\rfloor }2r_{2\ell+1}+\sum_{\ell=1}^{\left\lfloor \frac{k}{2}\right\rfloor }2r_{2\ell}\right]\\
 & =\frac{1}{2}\left[2r_{2j^{*}+1}+\sum_{1\leq\ell\leq\left\lfloor \frac{k}{2}\right\rfloor ,\ell\neq j^{*}}2r_{2\ell+1}+\sum_{\ell=1}^{\left\lfloor \frac{k}{2}\right\rfloor }2r_{2\ell}\right]\\
 & \leq\frac{1}{2}\left[\left(2j^{*}+1\right)r_{2j^{*}+1}+\sum_{1\leq\ell\leq\left\lfloor \frac{k}{2}\right\rfloor ,\ell\neq j^{*}}2r_{2\ell+1}+\sum_{\ell=1}^{\left\lfloor \frac{k}{2}\right\rfloor }2r_{2\ell}\right]-\frac{1}{2}r_{2j^{*}+1}\\
 & \leq\frac{1}{2}\left[\left(2j^{*}+1\right)r_{2j^{*}+1}+\sum_{1\leq\ell\leq\left\lfloor \frac{k}{2}\right\rfloor ,\ell\neq j^{*}}2r_{2\ell+1}+\sum_{\ell=1}^{\left\lfloor \frac{k}{2}\right\rfloor }2r_{2\ell}\right]-\frac{1}{2}\\
 & \leq\frac{1}{2}\left[\sum_{\ell=1}^{\left\lfloor \frac{k}{2}\right\rfloor }\left(2\ell+1\right)r_{2\ell+1}+\sum_{\ell=1}^{\left\lfloor \frac{k}{2}\right\rfloor }2\ell r_{2\ell}\right]-\frac{1}{2}\\
 & =\frac{k}{2}-\frac{1}{2}.
\end{align*}
Now suppose that $\ell^{*}=2j^{*}$ for some $j^{*}\geq2$. Then

\begin{align*}
\sum_{\ell=1}^{k}r_{\ell} & =\frac{1}{2}\left[\sum_{\ell=1}^{\left\lfloor \frac{k}{2}\right\rfloor }2r_{2\ell+1}+\sum_{\ell=1}^{\left\lfloor \frac{k}{2}\right\rfloor }2r_{2\ell}\right]\\
 & =\frac{1}{2}\left[2r_{2j^{*}}+\sum_{\ell=1}^{\left\lfloor \frac{k}{2}\right\rfloor }2r_{2\ell+1}+\sum_{1\leq\ell\leq\left\lfloor \frac{k}{2}\right\rfloor ,\ell\neq j^{*}}2r_{2\ell}\right]\\
 & \leq\frac{1}{2}\left[2j^{*}r_{2j^{*}}+\sum_{\ell=1}^{\left\lfloor \frac{k}{2}\right\rfloor }2r_{2\ell+1}+\sum_{\ell=1}^{\left\lfloor \frac{k}{2}\right\rfloor }2r_{2\ell}\right]-r_{2j^{*}}\\
 & \leq\frac{1}{2}\left[2j^{*}r_{2j^{*}}+\sum_{\ell=1}^{\left\lfloor \frac{k}{2}\right\rfloor }2r_{2\ell+1}+\sum_{\ell=1}^{\left\lfloor \frac{k}{2}\right\rfloor }2r_{2\ell}\right]-1\\
 & =\frac{1}{2}\left[\sum_{\ell=1}^{\left\lfloor \frac{k}{2}\right\rfloor }\left(2\ell+1\right)r_{2\ell+1}+\sum_{\ell=1}^{\left\lfloor \frac{k}{2}\right\rfloor }2\ell r_{2\ell}\right]-1\\
 & =\frac{k}{2}-1.
\end{align*}
So we have the upper bound $\frac{k}{2}-\frac{1}{2}$ for $\sum_{\ell=1}^{k}r_{\ell}$
independently of the parity of $\ell^{*}$. Hence,
\begin{align}
\frac{1}{n^{\frac{k}{2}}}\frac{n!}{r_{1}!\cdots r_{k}!\left(n-\sum_{\ell=1}^{k}r_{\ell}\right)!}\frac{k!}{1!^{r_{1}}\cdots k!^{r_{k}}}\IE\,U\left(\underbar{\ensuremath{\boldsymbol{r}}}\right) & \leq C\frac{1}{n^{\frac{k}{2}}}\frac{n!}{\left(n-\sum_{\ell=1}^{k}r_{\ell}\right)!}\nonumber \\
 & \leq C\frac{1}{n^{\frac{k}{2}}}n^{\sum_{\ell=1}^{k}r_{\ell}}\nonumber \\
 & \leq C\frac{1}{\sqrt{n}}.\label{eq:odd_summands}
\end{align}
Finally, we set 
\[
\Pi_{1}\coloneq\left\{ \underbar{\ensuremath{\boldsymbol{r}}}\in\Pi\,|\,r_{\ell}=0,\ell\neq2\right\} .
\]
Note that $\Pi_{1}=\emptyset$ if $k$ is odd and $\Pi_{1}=\left\{ \left(0,\frac{k}{2},0,\ldots,0\right)\right\} $
if $k$ is even. This means by (\ref{eq:single_ind}) and (\ref{eq:odd_summands})
that 
\[
\IE\left(\frac{1}{\sqrt{n}}\sum_{i=1}^{n}U_{i}\right)^{k}=\frac{1}{n^{\frac{k}{2}}}\sum_{\underbar{\ensuremath{\boldsymbol{r}}}\in\Pi}\frac{n!}{r_{1}!\cdots r_{k}!\left(n-\sum_{\ell=1}^{k}r_{\ell}\right)!}\frac{k!}{1!^{r_{1}}\cdots k!^{r_{k}}}\IE\,U\left(\underbar{\ensuremath{\boldsymbol{r}}}\right)\leq C\frac{1}{\sqrt{n}}
\]
if $k$ is odd. Now let $k$ be even, and hence $\Pi_{1}=\left\{ \left(0,\frac{k}{2},0,\ldots,0\right)\right\} $.
Then
\begin{align*}
\IE\left(\frac{1}{\sqrt{n}}\sum_{i=1}^{n}U_{i}\right)^{k} & \approx\frac{1}{n^{\frac{k}{2}}}\frac{n^{\sum_{\ell=1}^{k}r_{\ell}}}{\left(\frac{k}{2}\right)!}\frac{k!}{\left(2!\right)^{\frac{k}{2}}}\IE\,U_{1}^{2}\cdots U_{\frac{k}{2}}^{2}\\
 & =\frac{1}{n^{\frac{k}{2}}}\frac{n^{\frac{k}{2}}}{\left(\frac{k}{2}\right)!}\frac{k!}{2^{\frac{k}{2}}}\sigma^{k}\\
 & =\left(k-1\right)!!\,\sigma^{k},
\end{align*}
where we used the identity $\frac{k!}{\left(\frac{k}{2}\right)!2^{\frac{k}{2}}}=\left(k-1\right)!!$
for even $k$. The last expression above is the moment of order $k$,
for $k$ even, of the distribution $\mathcal{N}\left(0,\sigma^{2}\right)$.
The odd moments of $\mathcal{N}\left(0,\sigma^{2}\right)$ are 0 and
by (\ref{eq:single_ind}) and (\ref{eq:odd_summands}) for odd $k$
all summands of (\ref{eq:moment_summands}) converge to 0. This concludes
the proof.
\end{proof}
The previous lemma serves to determine the limiting standard error
of the statistic $T$.
\begin{prop}
The standard error of the statistic $T$ satisfies
\[
\lim_{n\rightarrow\infty}\sqrt{n}\sqrt{\IV\,T}=\sqrt{\IV_{\beta,N}S^{2}}.
\]
\end{prop}

\begin{proof}
This follows directly from Theorem \ref{thm:properties_bML_fin} and
Lemma \ref{lem:conv_mom}.
\end{proof}
While the statistic $T$ has a finite standard deviation for all population
sizes $N$ and all sample sizes $n$, the same is not true for the
estimator $\hat{\beta}_{N}$ as this estimator assumes the values
$\pm\infty$ with positive probability for all $N$ and all $n$.
Thus a similar statement does not hold for finite $n$ nor in the
limit $n\rightarrow\infty$. However, instead of using Lemma \ref{lem:conv_mom}
to determine the limiting standard error, we can employ statement
2 of Theorem \ref{thm:properties_bML_fin} to obtain the limiting
probabilities for a certain class of events. Below, $\mathcal{N}\left(0,\sigma^{2}\right)A$
stands for the probability of any measurable set $A$ under the $\mathcal{N}\left(0,\sigma^{2}\right)$
distribution.
\begin{prop}
Let $A\subset\IR$ be a Lebesgue measurable set. Then we have
\[
\IP\left\{ \hat{\beta}_{N}\in\beta+\frac{1}{\sqrt{n}}A\right\} \xrightarrow[n\rightarrow\infty]{}\mathcal{N}\left(0,\frac{4N^{2}}{\IV_{\beta,N}S^{2}}\right)A.
\]
In particular, we have for any $\varepsilon>0$, $\IP\left\{ \left|\hat{\beta}_{N}-\beta\right|\geq\frac{\varepsilon}{\sqrt{n}}\right\} \xrightarrow[n\rightarrow\infty]{}\mathcal{N}\left(0,\frac{4N^{2}}{\IV_{\beta,N}S^{2}}\right)\left(-\varepsilon,\varepsilon\right)^{c}$.
\end{prop}

\begin{proof}
Theorem \ref{thm:properties_bML_fin} states that
\[
\sqrt{n}\left(\hat{\beta}_{N}-\beta\right)\xrightarrow[n\rightarrow\infty]{\textup{d}}\mathcal{N}\left(0,\frac{4N^{2}}{\IV_{\beta,N}S^{2}}\right).
\]
By the definition of convergence in distribution and the absolute
continuity of the normal distribution, we have for any measurable
set $A$,
\begin{align*}
\IP\left\{ \hat{\beta}_{N}\in\beta+\frac{1}{\sqrt{n}}A\right\}  & =\IP\left\{ \hat{\beta}_{N}-\beta\in\frac{1}{\sqrt{n}}A\right\} \\
 & =\IP\left\{ \sqrt{n}\left(\hat{\beta}_{N}-\beta\right)\in A\right\} \\
 & \xrightarrow[n\rightarrow\infty]{}\mathcal{N}\left(0,\frac{4N^{2}}{\IV_{\beta,N}S^{2}}\right)A.
\end{align*}
\end{proof}
Even in the absence of a finite standard deviation for the estimator
$\hat{\beta}_{N}$, the previous proposition gives us an explicit
limit for the probability of a deviation of order $\frac{1}{\sqrt{n}}$
of the estimator $\hat{\beta}_{N}$ from the true parameter value
$\beta$.

\section{\label{sec:Optimal-Weights}Optimal Weights in a Two-Tier Voting
System}

A two-tier voting system describes a scenario where the population
of a state or union of states is divided into $M$ groups (e.g., member
states). Each group sends a representative to a common council that
makes decisions for the union. These representatives cast their votes
(`yes' or `no') based on the majority opinion within their respective
group. Each group $\lambda\in\IN_{M}$ is assumed to be of size $N_{\lambda}\in\IN$.
The votes of individual voter are represented by the random variable
$X_{\lambda i}$, where $\lambda\in\IN_{M}$ indicates the group and
$i\in N_{\lambda}$ the individual within the group. Recall the group
voting margins $S_{\lambda}$ defined in (\ref{eq:S_lambda}), and
set
\[
\bar{S}\coloneq\sum_{\lambda=1}^{M}S_{\lambda}
\]
to be the overall voting margin.

Given the group voting margin $S_{\lambda}$, we define the council
vote of each group.
\begin{defn}
\label{def:council-vote}The council vote of each group $\lambda\in\IN_{M}$
is
\begin{align*}
\chi_{\lambda} & \coloneq\begin{cases}
1 & \textup{if }S_{\lambda}>0,\\
-1 & \textup{otherwise.}
\end{cases}
\end{align*}
\end{defn}

Since the groups may vary in size, it is natural to assign different
voting weights $w_{\lambda}$ to each representative, reflecting the
relative sizes of their groups. In some situations, the groups can
be formed in such a manner that they have similar sizes. For example,
when electing a parliament such as the U.S. House of Representatives,
the country is typically divided into districts with roughly equal
populations, each receiving one seat. This approach is practical within
a single country but becomes less feasible in other contexts. It would
be impractical to reassemble countries or member states (like those
in the United Nations or European Union) into districts of equal size
or equally populated groups due to sovereignty concerns. Thus, the
issue of how to assign voting weights to groups of different sizes
cannot be avoided.

The problem of determining these optimal weights involves minimising
the democracy deficit (see Definition \ref{def:democracy-deficit}
below), i.e.\! the deviation between a council vote and an idealised
referendum across the entire population. This concept was first explored
for binary voting by Felsenthal and Machover \cite{FelsMach1999},
and later analysed in various contexts by other authors (e.g., \cite{Ki2007,ZyczSlom2014,KirsLang2014,MaasNape2012,To2020phd,KirsToth2022,KT2021c}.
Informally, imposing the democracy deficit as a criterion, we endeavour
to assign the voting weights in the council in such fashion that,
on average, the council vote shall be as close as possible to a hypothetical
referendum.

Other approaches to optimal weights are based on welfare considerations
or the goal of equalising the influence of all voters within the overall
population. The seminal work by Penrose \cite{Pen1946}, which introduced
the square root law as a rule for assigning voting weights that equalises
each voter's probability of being decisive in a two-tier system under
the assumption of independent voting, exemplifies this approach. Further
contributions to understanding optimal voting weights from welfare
and influence perspectives can be found in \cite{BeisBove2007,KoMaTrLa2013,KurMaaNa2017}.

In order to define the democracy deficit, we need a voting model that
describes voting behaviour across the overall population. We will
assume that the overall population behaves according to a Curie-Weiss
model $\IP_{\boldsymbol{\beta},\boldsymbol{N}}$ (see Definition \ref{def:CWM})
for some fixed $\boldsymbol{\beta}=\left(\beta_{1},\ldots,\beta_{M}\right)\geq0$
and $\boldsymbol{N}=\left(N_{1},\ldots,N_{M}\right)\in\IN^{M}$. With
a voting model in place, we can proceed to define the democracy deficit,
which will serve as a criterion for the determination of the optimal
voting weights each group receives in the council.
\begin{defn}
\label{def:democracy-deficit}The democracy deficit for a set of voting
weights $w_{1},\ldots,w_{M}\in\IR$ is given by
\[
\IE_{\boldsymbol{\beta},\boldsymbol{N}}\left[\bar{S}-\sum_{\lambda=1}^{M}w_{\lambda}\chi_{\lambda}\right]^{2}.
\]
We will call any vector $\left(w_{1},\ldots,w_{M}\right)\in\IR$ of
weights which minimises the democracy deficit `optimal'.
\end{defn}

\begin{prop}
\label{prop:opt_weights}For all $\boldsymbol{\beta}=\left(\beta_{1},\ldots,\beta_{M}\right)\geq0$,
the \emph{optimal weights} are given by
\[
w_{\lambda}=\IE_{\beta_{\lambda},N_{\lambda}}\left|S_{\lambda}\right|,\quad\lambda\in\IN_{M}.
\]
\end{prop}

\begin{proof}
This result was first proved in \cite{Ki2007}. For the reader's convenience,
we present a short proof here.

We find the minimum of the expression defining the democracy deficit
by deriving\\
$\IE_{\boldsymbol{\beta},\boldsymbol{N}}\left[S-\sum_{\lambda=1}^{M}w_{\lambda}\chi_{\lambda}\right]^{2}$
with respect to each $w_{\lambda}$, $\lambda\in\IN_{M}$, and equating
each derivative to 0:
\[
\IE_{\boldsymbol{\beta},\boldsymbol{N}}\left[\left(\bar{S}-\sum_{\nu=1}^{M}w_{\nu}\chi_{\nu}\right)\chi_{\lambda}\right]=0,
\]
which is equivalent to
\[
\IE_{\boldsymbol{\beta},\boldsymbol{N}}\,\bar{S}\,\chi_{\lambda}=\sum_{\nu=1}^{M}w_{\nu}\IE_{\boldsymbol{\beta},\boldsymbol{N}}\chi_{\nu}\chi_{\lambda}.
\]
Due to Definitions \ref{def:CWM} and \ref{def:council-vote}, we
have
\begin{align*}
\IE_{\boldsymbol{\beta},\boldsymbol{N}}\,\bar{S}\,\chi_{\lambda} & =\sum_{\nu=1}^{M}\IE_{\boldsymbol{\beta},\boldsymbol{N}}\,S_{\nu}\,\chi_{\lambda}\\
 & =\IE_{\beta_{\lambda},N_{\lambda}}\,S_{\lambda}\,\chi_{\lambda}\\
 & =\IE_{\beta_{\lambda},N_{\lambda}}\left|S_{\lambda}\right|
\end{align*}
and
\begin{align*}
\sum_{\nu=1}^{M}w_{\nu}\,\IE_{\boldsymbol{\beta},\boldsymbol{N}}S_{\nu}\chi_{\lambda} & =w_{\lambda}\,\IE_{\beta_{\lambda},N_{\lambda}}\chi_{\lambda}^{2}\\
 & =w_{\lambda}.
\end{align*}
The optimality condition is therefore
\[
w_{\lambda}=\IE_{\beta_{\lambda},N_{\lambda}}\left|S_{\lambda}\right|,\quad\lambda\in\IN_{M}.
\]
\end{proof}
We will now show that the expectation $\IE_{\beta_{\lambda},N_{\lambda}}\left|S\right|$
of the voting margin of a single group is strictly increasing in the
respective coupling constant $\beta_{\lambda}$ (cf. Proposition \ref{prop:ES2_fin}).
The following is an auxiliary lemma for this purpose.
\begin{lem}
\label{lem:expect_ineq}Let $X$ and $Y$ be random variables that
take values in a countable set $E\subset\IR$. Let $c\in\IR$ and
define the sets $A\coloneq E\cap\left(-\infty,c\right)$, $B\coloneq E\cap\left[c,\infty\right)$.
Assume that $A,B\neq\emptyset$ and
\begin{align*}
\IP\left\{ X=x\right\}  & >\IP\left\{ Y=x\right\} ,\quad x\in A,\\
\IP\left\{ X=x\right\}  & <\IP\left\{ Y=x\right\} ,\quad x\in B.
\end{align*}
Then $\IE\,Y>\IE\,X$ holds.
\end{lem}

\begin{proof}
We define the constant
\begin{equation}
t\coloneq\IP\left\{ X\in A\right\} -\IP\left\{ Y\in A\right\} .\label{eq:t}
\end{equation}
The constant $t$ can also be expressed as
\begin{equation}
t=1-\IP\left\{ X\in B\right\} -\left(1-\IP\left\{ Y\in B\right\} \right)=\IP\left\{ Y\in B\right\} -\IP\left\{ X\in B\right\} .\label{eq:t2}
\end{equation}
We write
\begin{align}
\IE\,Y-\IE\,X & =\sum_{x\in E}x\left(\IP\left\{ Y=x\right\} -\IP\left\{ X=x\right\} \right)\nonumber \\
 & =\sum_{x\in A}x\left(\IP\left\{ Y=x\right\} -\IP\left\{ X=x\right\} \right)+\sum_{x\in B}x\left(\IP\left\{ Y=x\right\} -\IP\left\{ X=x\right\} \right).\label{eq:diff_expec_sum}
\end{align}
For all $x\in A$, $x<c$ and $\IP\left\{ Y=x\right\} -\IP\left\{ X=x\right\} <0$
hold. Thus, for the first summand in (\ref{eq:diff_expec_sum}), we
have the lower bound
\begin{align}
\sum_{x\in A}x\left(\IP\left\{ Y=x\right\} -\IP\left\{ X=x\right\} \right) & >c\sum_{x\in A}\left(\IP\left\{ Y=x\right\} -\IP\left\{ X=x\right\} \right)\nonumber \\
 & =c\left(\IP\left\{ Y\in A\right\} -\IP\left\{ X\in A\right\} \right)\nonumber \\
 & =-ct,\label{eq:LB1}
\end{align}
where in the last step we used (\ref{eq:t}).

For all $x\in B$, $x\geq c$ and $\IP\left\{ Y=x\right\} -\IP\left\{ X=x\right\} >0$
hold. Thus, for the second summand in (\ref{eq:diff_expec_sum}),
we have the bound
\begin{align}
\sum_{x\in B}x\left(\IP\left\{ Y=x\right\} -\IP\left\{ X=x\right\} \right) & \geq c\sum_{x\in B}\left(\IP\left\{ Y=x\right\} -\IP\left\{ X=x\right\} \right)\nonumber \\
 & =c\left(\IP\left\{ Y\in B\right\} -\IP\left\{ X\in B\right\} \right)\nonumber \\
 & =ct,\label{eq:LB2}
\end{align}
where in the last step we used (\ref{eq:t2}).

Combining the lower bounds in (\ref{eq:LB1}) and (\ref{eq:LB2})
yields the claim due to (\ref{eq:diff_expec_sum}).
\end{proof}
\begin{prop}
\label{prop:E|S|_increasing}For fixed $N\in\IN$, the function $\beta\in\IR\mapsto\IE_{\beta,N}\left|S\right|\in\IR$
is strictly increasing and continuous.
\end{prop}

\begin{proof}
The continuity is immediate. We show that $\beta\mapsto\IE_{\beta,N}\left|S\right|$
is strictly increasing. For this purpose, we calculate the derivative
of $p_{\beta}\left(x\right)/Z_{\beta,N}$ for any $x\in\Omega_{N}$
employing Lemma \ref{lem:Z_deriv}. Recall that $p_{\beta}\left(x\right)$
equals $\exp\left(\frac{\beta}{2N}\left(\sum_{i=1}^{N}x_{i}\right)^{2}\right)$
and $s\left(x\right)$ stands for $\sum_{i=1}^{N}x_{i}$ for all $x\in\Omega_{N}$.
\begin{align}
\frac{\textup{d}}{\textup{d}\beta}\left(\frac{p_{\beta}\left(x\right)}{Z_{\beta,N}}\right) & =\frac{\frac{\textup{d}p_{\beta}\left(x\right)}{\textup{d}\beta}Z_{\beta,N}-p_{\beta}\left(x\right)\frac{\textup{d}Z_{\beta,N}}{\textup{d}\beta}}{Z_{\beta,N}^{2}}\nonumber \\
 & =\frac{\frac{s\left(x\right)^{2}}{2N}p_{\beta}\left(x\right)}{Z_{\beta,N}}-\frac{p_{\beta}\left(x\right)\frac{Z_{\beta,N}}{2N}\IE_{\beta,N}S^{2}}{Z_{\beta,N}^{2}}\nonumber \\
 & =\frac{p_{\beta}\left(x\right)}{2NZ_{\beta,N}}\left(s\left(x\right)^{2}-\IE_{\beta,N}S^{2}\right).\label{eq:px_Z_deriv}
\end{align}
We note that the derivative is positive if and only if
\begin{equation}
s\left(x\right)^{2}>\IE_{\beta,N}S^{2}.\label{eq:cond_pos_deriv}
\end{equation}
We define the set
\[
G\coloneq\left\{ \ell^{2}\,|\,\ell\in\IN,\;\ell=N\!\!\!\!\mod 2,\;m<\ell\leq N\right\} ,
\]
and let
\[
G=\left\{ g_{1},\ldots,g_{\left|G\right|}\right\} 
\]
be an enumeration of $G$ in ascending order. By Proposition \ref{prop:ES2_fin},
$\beta\in\IR\mapsto\IE_{\beta,N}S^{2}\in\left(\kappa,N^{2}\right)$
is continuous and strictly increasing. Hence, the function is injective.
By Lemma \ref{lem:ES2_outside}, it is surjective.

We define the constants $b_{i}\in\IR\cup\left\{ \infty\right\} $,
$i\in\IN_{\left|G\right|}$, by the condition
\[
\IE_{b_{i},N}S^{2}=g_{i}.
\]
Note that $\left(b_{i}\right)_{i\in\IN_{\left|G\right|}}$ is a strictly
increasing (since $\left(g_{i}\right)_{i\in\IN_{\left|G\right|}}$
is strictly increasing) finite sequence, and $b_{\left|G\right|}=\infty$
due to $g_{\left|G\right|}=N^{2}$ and $\lim_{\beta\rightarrow\infty}\IE_{\beta,N}S^{2}=N^{2}$
by Lemma \ref{lem:ES2_outside}. Using these constants, we define
the sets
\begin{align*}
B_{1} & \coloneq\left(-\infty,b_{1}\right),\\
B_{i} & \coloneq\left[b_{i-1},b_{i}\right),\quad i\in\IN_{\left|G\right|}\backslash\left\{ 1\right\} .
\end{align*}
Due to the bijectivity of $\beta\in\IR\mapsto\IE_{\beta,N}S^{2}\in\left(\kappa,N^{2}\right)$,
$B_{1},\ldots,B_{\left|G\right|}$ is a partition of $\IR$.

With these preparations, we first show that for all $\beta_{1}<\beta_{2}$
with $\beta_{1},\beta_{2}\in B_{i}$ for some $i\in\IN_{\left|G\right|}$,
$\IE_{\beta_{1},N}\left|S\right|<\IE_{\beta_{2},N}\left|S\right|$
holds. We define the following subsets of $\Omega_{N}$:
\begin{align*}
A_{r} & \coloneq\left\{ x\in\Omega_{N}\,\left|\,s\left(x\right)^{2}\leq\IE_{\beta_{r},N}S^{2}\right.\right\} ,\quad r=1,2,
\end{align*}
and write $A^{c}$ for the complement of any subset $A$ of $\Omega_{N}$.
By the definition of the set $G$ and the sets $B_{j}$, $j\in\IN_{\left|G\right|}$,
the equality $A_{1}=A_{2}$ is satisfied. We set
\[
A\coloneq A_{1}.
\]
We use the derivatives of $p_{\beta}\left(x\right)/Z_{\beta,N}$
in (\ref{eq:px_Z_deriv}) and the positivity condition (\ref{eq:cond_pos_deriv})
for said derivatives.

Since $\beta\in\IR\mapsto\IE_{\beta,N}S^{2}\in\left(m,N^{2}\right)$
is strictly increasing, the sign of the derivative of $p_{\beta}\left(\cdot\right)/Z_{\beta,N}$,
for any $\beta\in\left(\beta_{1},\beta_{2}\right)$ and any $x\in\Omega_{N}$
is
\begin{align*}
\frac{p_{\beta}\left(x\right)}{2NZ_{\beta,N}}\left(s\left(x\right)^{2}-\IE_{\beta,N}S^{2}\right) & <0,\quad x\in A,\\
\frac{p_{\beta}\left(x\right)}{2NZ_{\beta,N}}\left(s\left(x\right)^{2}-\IE_{\beta,N}S^{2}\right) & >0,\quad x\in A^{c}.
\end{align*}
These signs and the fundamental theorem of calculus (note that the
derivatives of $p_{\beta}\left(x\right)/Z_{\beta,N}$ are continuous)
yield for all $x\in A$
\begin{align*}
\frac{p_{\beta_{2}}\left(x\right)}{Z_{\beta_{2},N}} & =\frac{p_{\beta_{1}}\left(x\right)}{Z_{\beta_{1},N}}+\int_{\beta_{1}}^{\beta_{2}}\frac{\textup{d}}{\textup{d}\beta}\left(\frac{p_{\beta}\left(x\right)}{Z_{\beta,N}}\right)\,\textup{d}\beta<\frac{p_{\beta_{1}}\left(x\right)}{Z_{\beta_{1},N}}
\end{align*}
and for all $x\in A^{c}$
\begin{align*}
\frac{p_{\beta_{2}}\left(x\right)}{Z_{\beta_{2},N}} & =\frac{p_{\beta_{1}}\left(x\right)}{Z_{\beta_{1},N}}+\int_{\beta_{1}}^{\beta_{2}}\frac{\textup{d}}{\textup{d}\beta}\left(\frac{p_{\beta}\left(x\right)}{Z_{\beta,N}}\right)\,\textup{d}\beta>\frac{p_{\beta_{1}}\left(x\right)}{Z_{\beta_{1},N}}.
\end{align*}
We apply Lemma \ref{lem:expect_ineq} with $X$ being $\left|S\right|$
following the distribution $\IP_{\beta_{1},N}\circ\left|S\right|^{-1}$
and $Y$ being $\left|S\right|$ following the distribution $\IP_{\beta_{2},N}\circ\left|S\right|^{-1}$.
Then the statement $\IE_{\beta_{1},N}\left|S\right|<\IE_{\beta_{2},N}\left|S\right|$
follows by Lemma \ref{lem:expect_ineq}.

Now assume that $\beta_{1},\beta_{2}\in\IR$, $\beta_{1}<\beta_{2}$,
and $\beta_{1}\in B_{j_{1}}$, $\beta_{2}\in B_{j_{2}}$ for some
$j_{1}<j_{2}\in\IN_{\left|G\right|}$. Since $\beta\mapsto\IE_{\beta,N}\left|S\right|$
is continuous and strictly increasing on each $B_{i}$, $i\in\IN_{\left|G\right|}$,
we obtain for all $\beta\in B_{i}$
\[
\IE_{b_{i},N}\left|S\right|=\lim_{\beta'\nearrow b_{i}}\IE_{\beta',N}\left|S\right|>\IE_{\beta,N}\left|S\right|
\]
and
\[
\IE_{b_{i+1},N}\left|S\right|=g_{i+1}>g_{i}=\IE_{b_{i},N}\left|S\right|.
\]
Taking into account $\beta_{1}\in B_{j_{1}}=\left[b_{j_{1}-1},b_{j_{1}}\right)$,
$\beta_{2}\in\left[b_{j_{2}-1},b_{j_{2}}\right)$, and $j_{2}-1\geq j_{1}$,
we therefore have
\[
\IE_{\beta_{2},N}\left|S\right|\geq\IE_{b_{j_{2}-1},N}\left|S\right|\geq\IE_{b_{j_{1}},N}\left|S\right|>\IE_{\beta_{1},N}\left|S\right|,
\]
and we have proved that the function $\beta\mapsto\IE_{\beta,N}\left|S\right|$
is strictly increasing.
\end{proof}
By Propositions \ref{prop:opt_weights} and \ref{prop:E|S|_increasing},
we have
\begin{cor}
Let $\lambda,\mu\in\IN_{M}$, $N_{\lambda}=N_{\mu}$, and $0\leq\beta_{\lambda}<\beta_{\mu}$.
Then the optimal weights $w_{\lambda}$ and $w_{\mu}$ satisfy the
inequality
\[
w_{\lambda}<w_{\mu}.
\]
\end{cor}

This implies that among groups of equal sizes, those groups with stronger
interactions between voters, that is groups with a higher parameter
$\beta$, will receive a larger weight in the council than groups
with voters who interact more loosely with each other, i.e.\! groups
that that have a lower parameter $\beta$. Put in a different way,
under the CWM as a voting model, there are two avenues for groups
to obtain a higher voting weight under the democracy deficit criterion:
have a larger population or be more cohesive in the group votes.

Proposition \ref{prop:opt_weights} yields an estimator for the optimal
weights of each group by substituting any estimator for the parameter
$\beta_{\lambda}$. For example, an estimator for the optimal weights
based on the estimator $\hat{\boldsymbol{\beta}}_{\boldsymbol{N}}$
from Definition \ref{def:estimator_fin_N} can be defined as follows.
Let for all $\lambda\in\IN_{M}$, $\hat{\beta}_{N_{\lambda}}\left(\lambda\right)\coloneq\left(\hat{\boldsymbol{\beta}}_{\boldsymbol{N}}\right)_{\lambda}$.
Then, $\hat{\beta}_{N_{\lambda}}\left(\lambda\right):\Omega_{N_{\lambda}}^{n}\rightarrow\left[-\infty,\infty\right]$
is an estimator for $\beta_{\lambda}$ calculated for a sample of
observations of voting in group $\lambda$.
\begin{defn}
\label{def:estimator_weights}Let $\lambda\in\IN_{M}$, $\beta_{\lambda}\geq0$.
The estimator $\hat{w}_{\lambda}:\Omega_{N_{\lambda}}^{n}\rightarrow\left[0,\infty\right)$
for the optimal weight of group $\lambda$ based on $\hat{\beta}_{N_{\lambda}}\left(\lambda\right)$
is defined by
\[
\hat{w}_{\lambda}=\IE_{\hat{\beta}_{N_{\lambda}}\left(\lambda\right),N_{\lambda}}\left|S_{\lambda}\right|.
\]
\end{defn}

The estimator $\hat{w}_{\lambda}$ inherits many of the properties
of $\hat{\boldsymbol{\beta}}_{\boldsymbol{N}}$. This is the subject
of the next theorem. Recall the good rate function $\boldsymbol{J}$
defined in (\ref{eq:bold_J}), and the single-group rate function
$J_{\lambda}:\left[-\infty,\infty\right]\rightarrow\left[0,\infty\right]$
in (\ref{eqJ}) for group $\lambda$. We define the good rate function
$H_{\lambda}:\IR\rightarrow\left[0,\infty\right]$ by
\[
H_{\lambda}\left(y\right)\coloneq\inf\left\{ J_{\lambda}\left(\beta\right)\,|\,\beta\in\left[-\infty,\infty\right],\IE_{\beta,N_{\lambda}}\left|S_{\lambda}\right|=y\right\} ,\quad y\in\IR.
\]

\begin{thm}
Let $\lambda\in\IN_{M}$ and $N_{\lambda}\in\IN$. Let $w_{\lambda}$
be the optimal weight from Proposition \ref{prop:opt_weights}. Then
the following statements hold:
\begin{enumerate}
\item $\hat{w}_{\lambda}\xrightarrow[n\rightarrow\infty]{\textup{p}}w_{\lambda}$.
\item $\sqrt{n}\left(\hat{w}_{\lambda}-w_{\lambda}\right)\xrightarrow[n\rightarrow\infty]{\textup{d}}\mathcal{N}\left(0,\upsilon^{2}\right)$
and
\[
\upsilon^{2}=\frac{\left(\IE_{\beta,N_{\lambda}}\left|S_{\lambda}\right|^{3}-\IE_{\beta,N_{\lambda}}\left|S_{\lambda}\right|\IE_{\beta,N_{\lambda}}S_{\lambda}^{2}\right)^{2}}{\IV_{\beta,N}S^{2}}.
\]
\item $\hat{w}_{\lambda}$ satisfies a large deviations principle with rate
$n$ and rate function $H_{\lambda}$. $H_{\lambda}$ has a unique
minimum at $w_{\lambda}$, and we have for each closed set $K\subset\IR$
that does not contain $w_{\lambda}$, $\inf_{y\in K}H_{\lambda}\left(y\right)>0$
and
\[
\IP\left\{ \hat{w}_{N}\in K\right\} \leq2\exp\left(-n\inf_{y\in K}H_{\lambda}\left(y\right)\right).
\]
\end{enumerate}
\end{thm}

\begin{proof}
This theorem is proved in close analogy to Theorem \ref{thm:properties_bML_fin}.
\end{proof}

\appendix

\section*{Appendix}

We present a number of concepts and auxiliary results we use. Recall
the expression $Z_{\boldsymbol{\beta},\boldsymbol{N}}$ from (\ref{eq:part_fn}).
\begin{lem}
\label{lem:Z_deriv}The first derivative of $Z_{\boldsymbol{\beta},\boldsymbol{N}}$
with respect to $\beta_{\lambda}$, $\lambda\in\IN_{M}$ is
\begin{align*}
\frac{\textup{d}Z_{\boldsymbol{\beta},\boldsymbol{N}}}{\textup{d}\beta_{\lambda}} & =\frac{Z_{\boldsymbol{\beta},\boldsymbol{N}}}{2N_{\lambda}}\,\IE_{\boldsymbol{\beta},\boldsymbol{N}}S_{\lambda}^{2}.
\end{align*}
\end{lem}

\begin{proof}
The derivative of the partition function $Z_{\boldsymbol{\beta},\boldsymbol{N}}$
with respect to $\beta_{\lambda}$ is
\begin{align*}
\frac{\textup{d}Z_{\boldsymbol{\beta},\boldsymbol{N}}}{\textup{d}\beta_{\lambda}} & =\sum_{x\in\Omega_{N_{1}+\cdots+N_{M}}}\frac{\textup{d}}{\textup{d}\beta_{\lambda}}\left[\exp\left(\frac{1}{2}\sum_{\lambda=1}^{M}\frac{\beta_{\lambda}}{N_{\lambda}}\left(\sum_{i=1}^{N_{\lambda}}x_{\lambda i}\right)^{2}\right)\right]\\
 & =\sum_{x\in\Omega_{N_{1}+\cdots+N_{M}}}\frac{1}{2N_{\lambda}}\left(\sum_{i=1}^{N_{\lambda}}x_{\lambda i}\right)^{2}\exp\left(\frac{1}{2}\sum_{\lambda=1}^{M}\frac{\beta_{\lambda}}{N_{\lambda}}\left(\sum_{i=1}^{N_{\lambda}}x_{\lambda i}\right)^{2}\right)\\
 & =\frac{Z_{\boldsymbol{\beta},\boldsymbol{N}}}{2N_{\lambda}}\,\IE_{\boldsymbol{\beta},\boldsymbol{N}}S_{\lambda}^{2}.
\end{align*}
\end{proof}
\begin{thm}[Slutsky]
\label{thm:slutsky}Let $\left(Y_{n}\right)_{n\in\IN}$ and $\left(Z_{n}\right)_{n\in\IN}$
be sequences of random variables, $Y$ a random variable, and $a\in\IR$
a constant such that $Y_{n}\xrightarrow[n\rightarrow\infty]{\textup{d}}Y$
and $Z_{n}\xrightarrow[n\rightarrow\infty]{\textup{p}}a$. Then
\[
Y_{n}+Z_{n}\xrightarrow[n\rightarrow\infty]{\textup{d}}Y+a\quad\text{and}\quad Y_{n}Z_{n}\xrightarrow[n\rightarrow\infty]{\textup{d}}aY.
\]
\end{thm}

\begin{proof}
These statements are Theorems 11.3 and 11.4 in \cite{Gut2013}.
\end{proof}
\begin{thm}[Continuous Mapping]
\label{thm:cont_mapping}Let $\left(Y_{n}\right)_{n\in\IN}$ be a
sequence of random variables and $Y$ a random variable, each of them
taking values in some subset $A\subset\IR$, such that $Y_{n}\xrightarrow[n\rightarrow\infty]{\textup{p}}Y$,
and let $g:A\rightarrow\IR$ be a continuous function. Then
\[
g\left(Y_{n}\right)\xrightarrow[n\rightarrow\infty]{\textup{p}}g\left(Y\right).
\]
\end{thm}

\begin{proof}
See Theorem 2.3 in \cite{VanderVaart1998}.
\end{proof}
\begin{thm}[Delta Method]
\label{thm:delta_method}Let $\left(Y_{n}\right)_{n\in\IN}$ be a
sequence of random variables such that $\IE\,Y_{n}=\mu\in\IR$ for
all $n\in\IN$ and $\sqrt{n}\left(Y_{n}-\mu\right)\xrightarrow[n\rightarrow\infty]{\textup{d}}\mathcal{N}\left(0,\sigma^{2}\right)$
for a constant $\sigma>0$. Let $f:D\rightarrow\IR$ be a continuously
differentiable function with domain $D\subset\IR$ such that $Y_{n}\in D$
for all $n\in\IN$. Assume $f'\left(\mu\right)\neq0$. Then
\[
\sqrt{n}\left(f\left(Y_{n}\right)-f\left(\mu\right)\right)\xrightarrow[n\rightarrow\infty]{\textup{d}}\mathcal{N}\left(0,\left(f'\left(\mu\right)\right)^{2}\sigma^{2}\right)
\]
is satisfied.
\end{thm}

\begin{proof}
Since this result is not found as frequently in textbooks, we present
a proof here for the convenience of the readers.

We Taylor expand the function $f$ around the point $\mu$:
\[
f\left(Y_{n}\right)=f\left(\mu\right)+\left(Y_{n}-\mu\right)f'\left(\xi_{n}\right)
\]
for some $\xi_{n}$ which lies between $\mu$ and $Y_{n}$. We can
rewrite the above as
\begin{equation}
\sqrt{n}\left(f\left(Y_{n}\right)-f\left(\mu\right)\right)=\sqrt{n}\left(Y_{n}-\mu\right)f'\left(\xi_{n}\right).\label{eq:conv_f}
\end{equation}
Due to the assumption $\sqrt{n}\left(Y_{n}-\mu\right)\xrightarrow[n\rightarrow\infty]{\textup{d}}\mathcal{N}\left(0,\sigma^{2}\right)$,
\[
\left|\xi_{n}-Y_{n}\right|\leq\left|\mu-Y_{n}\right|\xrightarrow[n\rightarrow\infty]{\textup{p}}0,
\]
and as $f'$ is continuous, Theorem \ref{thm:cont_mapping} implies
\[
f'\left(\xi_{n}\right)\xrightarrow[n\rightarrow\infty]{\textup{p}}f'\left(\mu\right).
\]
The last display, $\sqrt{n}\left(Y_{n}-\mu\right)\xrightarrow[n\rightarrow\infty]{\textup{d}}\mathcal{N}\left(0,\sigma^{2}\right)$,
(\ref{eq:conv_f}), and Theorem \ref{thm:slutsky} together yield
\[
\sqrt{n}\left(f\left(Y_{n}\right)-f\left(\mu\right)\right)\xrightarrow[n\rightarrow\infty]{\textup{d}}\mathcal{N}\left(0,\left(f'\left(\mu\right)\right)^{2}\sigma^{2}\right).
\]
\end{proof}
Recall Notation \ref{notation:infty} for the expressions $\left[-\infty,\infty\right]$
and $\left[0,\infty\right]$.
\begin{defn}
\label{def:LDP}Let $\left(P_{n}\right)_{n\in\IN}$ be a sequence
of probability measures on a metric space $\mathcal{X}$, let $\left(a_{n}\right)_{n\in\IN}$
be a sequence of positive numbers with $a_{n}\xrightarrow[n\rightarrow\infty]{}\infty$,
and let $I:\mathcal{X}\rightarrow\left[0,\infty\right]$ be a function.
If $I$ is lower semi-continuous, i.e.\! its level sets $\left\{ x\in\mathcal{X}\,|\,I\left(x\right)\leq\alpha\right\} $
are closed for each $\alpha\in\left[0,\infty\right)$, we call $I$
a rate function. If the level sets are compact in $\mathcal{X}$ for
each $\alpha\in\left[0,\infty\right)$, we call $I$ a good rate function.
If $I$ is a good rate function, and the two conditions
\begin{enumerate}
\item $\limsup_{n\rightarrow\infty}\frac{1}{a_{n}}\ln P_{n}K\le-\inf_{x\in K}I\left(x\right)$
for each closed set $K\subset\mathcal{X}$,
\item $\liminf_{n\rightarrow\infty}\frac{1}{a_{n}}\ln P_{n}G\geq-\inf_{x\in G}I\left(x\right)$
for each open set $G\subset\mathcal{X}$
\end{enumerate}
hold, then we say that the sequence $\left(P_{n}\right)_{n\in\IN}$
satisfies a large deviations principle with rate $a_{n}$ and rate
function $I$. If $\left(Y_{n}\right)_{n\in\IN}$ is a sequence of
random variables taking values in $\mathcal{X}$ such that, for each
$n\in\IN$, $Y_{n}$ follows the distribution $P_{n}$, we will also
say that $\left(Y_{n}\right)_{n\in\IN}$ satisfies a large deviations
principle with rate $a_{n}$ and rate function $I$.

In our applications of large deviations principles, the metric space
$\mathcal{X}$ will be $\IR$ or $\left[-\infty,\infty\right]$.
\end{defn}

We present a standard convergence result concerning a statistic employed
in the estimation of the parameter $\beta$. This can be used to then
demonstrate convergence to $\beta$ for the estimators presented in
this article. Recall Definition \ref{def:cumul_entropy} of the entropy
function of a distribution.
\begin{prop}
\label{prop:conv_stat}Let $n,N\in\IN$ and let $R:\Omega_{N}^{n}\rightarrow\IR$
be a statistic of the form
\[
R\left(x^{(1)},\ldots,x^{(n)}\right)\coloneq\frac{1}{n}\sum_{t=1}^{n}f\left(x^{(t)}\right),\quad\left(x^{(1)},\ldots,x^{(n)}\right)\in\Omega_{N}^{n},
\]
for some function $f:\Omega_{N}\rightarrow\IR$. Let $X$ be a random
vector on $\Omega_{N}$ with Curie-Weiss distribution according to
Definition \ref{def:CWM}. Let $\mu\coloneq\IE_{\beta,N}f\left(X\right)$,
$\sigma^{2}\coloneq\IV_{\beta,N}f\left(X\right)$, and $\Lambda_{f\left(X\right)}^{*}$
the entropy function of $f\left(X\right)$.

Then the following three statements hold:
\begin{enumerate}
\item A law of large numbers holds for the sequence $R\left(x^{(1)},\ldots,x^{(n)}\right)$:
\[
R\left(x^{(1)},\ldots,x^{(n)}\right)\xrightarrow[n\rightarrow\infty]{\textup{p}}\mu.
\]
\item A central limit theorem holds for the sequence $\sqrt{n}\left(R\left(x^{(1)},\ldots,x^{(n)}\right)-\mu\right)$:
\[
\sqrt{n}\left(R\left(x^{(1)},\ldots,x^{(n)}\right)-\mu\right)\xrightarrow[n\rightarrow\infty]{\textup{d}}\mathcal{N}\left(0,\sigma^{2}\right).
\]
\item A large deviations principle holds for the sequence $R\left(x^{(1)},\ldots,x^{(n)}\right)$
with rate $n$ and rate function $I:\IR\rightarrow\left[0,\infty\right)\cup\left\{ \infty\right\} $,
\[
I\left(x\right)\coloneq\Lambda_{f\left(X\right)}^{*}\left(x\right),\quad x\in\IR.
\]
\item If $I$ has a unique global minimum at $x_{0}\in\IR$, then for any
closed set $K\subset\IR$ that does not contain $x_{0}$ we have $\delta\coloneq\inf_{x\in K}I\left(x\right)>0$,
and
\[
\IP\left\{ R\in K\right\} \leq2\exp\left(-\delta n\right)
\]
holds for all $n\in\IN$.
\end{enumerate}
\end{prop}

\begin{proof}
Since $R$ is defined as a sum of i.i.d. random variables with existing
variance, the first three results can be found in many books about
probability theory, and we therefore omit their proof. The last statement
is somewhat less commonly found, hence we prove it here.

The random variable $f\left(X\right)$ is bounded and not almost surely
constant, so Lemma \ref{lem:cumulant_entropy} applies to $\Lambda_{f\left(X\right)}^{*}$.

Set
\begin{equation}
\delta\coloneq\inf\left\{ \Lambda_{f\left(X\right)}^{*}\left(x\right)\,|\,x\in K\right\} .\label{eq:delta-1}
\end{equation}
As $\IE_{\beta,N}f\left(X\right)=\mu\in K^{c}$ and $K^{c}$ is open,
there is some $\eta>0$ such that the open ball $B_{\eta}\left(\mu\right)$
with radius $\eta$ and centre $\mu$ is a subset of $K^{c}$. We
choose $\varepsilon_{r}\coloneq\sup\left\{ \eta>0\,|\,\left(\mu,\mu+\eta\right)\subset K^{c}\right\} $
and $\varepsilon_{l}\coloneq\sup\left\{ \eta>0\,|\,\left(\mu-\eta,\mu\right)\subset K^{c}\right\} $.
Then $G\coloneq\left(\mu-\varepsilon_{l},\mu+\varepsilon_{r}\right)\subset K^{c}$.
By statements 4 and 5 of Lemma \ref{lem:cumulant_entropy},
\[
\delta=\inf\left\{ \Lambda_{f\left(X\right)}^{*}\left(x\right)\,|\,x\in G^{c}\right\} =\min\left\{ \Lambda_{f\left(X\right)}^{*}\left(\mu-\varepsilon_{l}\right),\Lambda_{f\left(X\right)}^{*}\left(\mu+\varepsilon_{r}\right)\right\} >0.
\]
We write
\[
\IP\left\{ R\in K\right\} \leq\IP\left\{ R\in G^{c}\right\} =\IP\left\{ R\leq\mu-\varepsilon_{l}\right\} +\IP\left\{ R\geq\mu+\varepsilon_{r}\right\} .
\]

We use Markov's inequality to obtain for all $t\leq0$
\begin{align*}
\IP\left\{ R\leq\mu-\varepsilon_{l}\right\}  & \leq\IP\left\{ \exp\left(nt\left(R-\left(\mu-\varepsilon_{l}\right)\right)\right)\geq1\right\} \leq\IE\exp\left(nt\left(R-\left(\mu-\varepsilon_{l}\right)\right)\right)\\
 & =\exp\left(-nt\left(\mu-\varepsilon_{l}\right)\right)\prod_{r=1}^{n}\IE\exp\left(tf\left(X^{(r)}\right)\right)=\exp\left(-nt\left(\mu-\varepsilon_{l}\right)\right)\left[\IE\exp\left(tf\left(X\right)\right)\right]^{n}\\
 & =\exp\left(-nt\left(\mu-\varepsilon_{l}\right)\right)\exp\left(n\Lambda_{f\left(X\right)}\left(t\right)\right)=\exp\left(-n\left(\left(\mu-\varepsilon_{l}\right)t-\Lambda_{f\left(X\right)}\left(t\right)\right)\right).
\end{align*}
As this holds for all $t\leq0$ and we have $\mu-\varepsilon_{l}<\mu=\IE_{\beta,N}f\left(X\right)$,
we arrive at
\begin{equation}
\IP\left\{ R\leq\mu-\varepsilon_{l}\right\} \leq\exp\left(-n\Lambda_{f\left(X\right)}^{*}\left(\mu-\varepsilon_{l}\right)\right).\label{eq:left_UB-1}
\end{equation}
Similarly, we calculate the upper bound
\begin{equation}
\IP\left\{ R\geq\mu+\varepsilon_{r}\right\} \leq\exp\left(-n\Lambda_{f\left(X\right)}^{*}\left(\mu+\varepsilon_{r}\right)\right).\label{eq:right_UB-1}
\end{equation}
Combining (\ref{eq:left_UB-1}) and (\ref{eq:right_UB-1}) yields
\begin{align*}
\IP\left\{ R\in K\right\}  & \leq\exp\left(-n\Lambda_{f\left(X\right)}^{*}\left(\mu-\varepsilon_{l}\right)\right)+\exp\left(-n\Lambda_{f\left(X\right)}^{*}\left(\mu+\varepsilon_{r}\right)\right)\\
 & \leq2\exp\left(-\delta n\right).
\end{align*}
\end{proof}
\begin{thm}[Contraction Principle]
\label{thm:contr_princ}Let $\mathcal{X}$ and $\mathcal{Y}$ be
metric spaces. Let $\left(P_{n}\right)_{n\in\IN}$ be a sequence of
probability measures on $\mathcal{X}$, and let $f:D\rightarrow\mathcal{Y}$
be a continuous function with its domain $D\subset\mathcal{X}$ containing
the support of $P_{n}$ for each $n\in\IN$. Let $\left(a_{n}\right)_{n\in\IN}$
a sequence of positive numbers with $a_{n}\xrightarrow[n\rightarrow\infty]{}\infty$,
and $I:\mathcal{X}\rightarrow\left[0,\infty\right]$ a good rate function.
We define $J:\mathcal{Y}\rightarrow\left[0,\infty\right]$ by
\[
J\left(y\right)\coloneq\inf\left\{ I\left(x\right)\,|\,x\in D,f\left(x\right)=y\right\} ,\quad y\in\mathcal{Y}.
\]
Then $J$ is a good rate function. If $\left(P_{n}\right)_{n\in\IN}$
satisfies a large deviations principle with rate $a_{n}$ and rate
function $I$, then the sequence of push forward measures $\left(P_{n}\circ f^{-1}\right)_{n\in\IN}$
on $\mathcal{Y}$ satisfies a large deviations principle with rate
$a_{n}$ and rate function $J$.
\end{thm}

\begin{proof}
This result can be found, e.g., in \cite[Theorem 4.2.1]{DembZeit1998}.
\end{proof}

\section*{Acknowledgements}

M.\! B.\! is a Fellow and G.\! T.\! is a Candidate of the Sistema
Nacional de Investigadoras e Investigadores. G.\! T.\! was supported
by a Secihti (formerly Conahcyt) postdoctoral fellowship.

\bibliographystyle{plain}

\end{document}